\documentclass[12pt,a4]{amsart}

\usepackage{cite}
\usepackage{a4wide}
\usepackage{chemarrow}
\usepackage{graphicx}

\usepackage{amssymb,amsmath,amscd}
\usepackage{latexsym}
\usepackage{tkz-graph}
\tikzset{node distance=2cm, auto}

\newtheorem{them}{\indent Theorem}[section]
\newtheorem{lem}[them]{\indent Lemma}
\newtheorem{pro}[them]{\indent Proposition}
\newtheorem{coro}[them]{\indent Corollary}
\newtheorem{ex}[them]{\indent Example}
\newtheorem{definition}[them]{\indent Definition}

\newcommand{\ig}{\operatorname{IG}}

\newcommand{\ar}{\mathcal{R}}
\newcommand{\el}{\mathcal{L}}

\newcommand{\dee}{\mbox{$\mathcal{D}$}}

\newcommand{\els}{\mbox{${\mathcal L}^{\ast}$}}

\newcommand{\ars}{\mbox{${\mathcal R}^{\ast}$}}
\newcommand{\art}{\mbox{$\widetilde{\mathcal{R}}$}}
\newcommand{\elt}{\mbox{$\widetilde{\mathcal{L}}$}}

\begin{document}
\title{Free idempotent generated semigroups over bands}
\keywords{biordered set, band, (weakly) abundant semigroup}
\date{\today}

\author{Victoria Gould}
\email{victoria.gould@york.ac.uk}
\author{Dandan Yang}
\email{ddy501@york.ac.uk}
\address{Department of Mathematics\\University
  of York\\Heslington\\York YO10 5DD\\UK}
  \subjclass[2010]{Primary 20M05;  20M10}

\begin{abstract} Free idempotent generated semigroups $\ig(E)$, where  $E$ is a biordered set, have provided a focus of recent research, the majority of the efforts concentrating on the behaviour of the maximal subgroups. Inspired by an example of Brittenham, Margolis and Meakin, several proofs have been offered that any group occurs as a maximal subgroup of some $\ig(E)$, the most recent being that of Dolinka and Ru\v{s}kuc, who show that $E$ can be taken to be a band.  From a result of  Easdown, Sapir and Volkov,  periodic elements  of any $\ig(E)$ lie in subgroups. However, little else is known  of the `global' properties of  $\ig(E)$, other than that it need not be regular, even where $E$ is a semilattice.

Since its introduction by Fountain in the late 1970s, the study of abundant and related  semigroups has given rise to  a  deep and fruitful research area. The classes of abundant and adequate semigroups extend those of regular and inverse semigroups, respectively, and themselves are contained in the classes of weakly abundant and weakly adequate semigroups. Recent significant developments include the description by Kambites, using birooted labelled trees,  of the free semigroups in the quasi-variety of adequate semigroups. 

Our main result shows that  for {\em any} band $B$, the semigroup $\ig(B)$ is  a weakly abundant semigroup and moreover satisfies a natural condition called the {\em congruence condition}.   We   show that if $B$ is a band for which $uv=vu=v$ for all $u,v\in B$ with $BvB\subset BuB$ (a condition certainly satisfied for semilattices),  then $\ig(B)$ is abundant with solvable word problem. Further, $\ig(B)$ is also abundant for a normal band $B$ for which $\ig(B)$ satisfies a given technical condition, and we give examples of such $B$.  On the other hand, we give an example of a normal band $B$ such that $\ig(B)$ is not abundant.
\end{abstract}
\maketitle
\section{Introduction}\label{sec:intro}

Let $S$ be a semigroup with  set of idempotents $E=E(S)$. It is easy to see that if idempotents of $S$ commute, then $E$ may be endowed with a partial order under which it becomes a semilattice, that is, every pair of elements has a greatest lower bound, which is just their product in $S$. For an arbitrary semigroup $S$, the  set $E$,  equipped with the restriction of the quasi-orders $\leq_{\mathcal{R}}$ and $\leq_\mathcal{L}$ defined on $S$, forms  a {\it biordered set} \cite{nambooripad:1979}. On the other hand, Easdown \cite{easdown:1985} shows every biordered set $E$ occurs as $E(S)$ for some semigroup $S.$ 

Given a biordered set $E$, which we can without prejudice take as the set $E$ of idempotents of some semigroup $S$, there is a free object in the category of semigroups that are generated by $E$. This is  called the {\it free idempotent generated semigroup} over $E$, and is given by the following presentation:
\[\ig(E)=\langle \overline{E}:  \bar{e}\bar{f}=\overline{ef},\, e,f\in E, \{ e,f\}
\cap \{ ef,fe\}\neq \emptyset\rangle,\]
where $\overline{E}=\{ \bar{e}:e\in E\}$.\footnote{It is more usual to identify elements of
$E$ with those of $\overline{E}$, but it helps the clarity of our later arguments to
make this distinction.} Note that $\{ e,f\}
\cap \{ ef,fe\}\neq \emptyset$ implies both $ef$ and $fe$ are idempotents of $E$; they  are referred to as {\em basic} products. Clearly, there is a natural morphism $\varphi$ from $\ig(E)$ to $\langle E\rangle$, the subsemigroup of $S$ generated by $E$. In fact, $E(\ig(E))=\overline{E}$, and the restriction $\varphi|_{\overline{E}}: \overline{E}\longrightarrow E$ is an isomorphism of biordered sets \cite{easdown:1985}. We refer our readers to \cite{gray:2012} for other classical properties of $\ig(E)$.

Given the universal nature of free idempotent generated semigroups, it is natural to enquire into their structure. A popular theme is to investigate their maximal subgroups, facilitated by the fact that {\em regular} $\dee$-classes of $\ig(E)$ have an `egg-box' structure corresponding to that in $S$ (see \cite{gray:2012}). Motivated by an example by  Brittenham, Margolis and Meakin \cite{brittenham:2009}, it was proved, first by Gray and Ru\v{s}kuc \cite{gray:2012} and later by the authors \cite{gould:2012}, that {\it every} group is a maximal subgroup of $\ig(E)$ for some biordered set $E$.  Dolinka and  Ru\v{s}kuc show that  $E$ may be taken to be a {\em band} (that is, a {semigroup} of idempotents) \cite{dolinka:2013}, thus, in particular, demonstrating the signficance of bands in the study  of free idempotent generated semigroups.

Whereas a deal of energy has recently been put into the question of the maximal subgroups of free idempotent generated semigroups $\ig(E)$, in contrast, very little is known of the overall structure of semigroups of this form. What can be said is that periodic elements of $\ig(E)$ must lie in subgroups, a result of Easdown, Sapir and Volkov \cite{easdown:2010}, and that $\ig(E)$ need not be regular. Indeed, even for a semilattice $Y$,  the semigroup $\ig(Y)$ need not be regular \cite[Example 2]{brittenham:2009}.   Regularity is a property of semigroups that can be phrased in terms of Green's relations $\ar$ and $\el$ and idempotents. Analogous but weaker conditions are those of being {\em abundant} and {\em weakly abundant}, which are defined in the same way but with $\ar$ and $\el$ replaced by $\ars$ and $\els$, or $\art$ and $\elt$, respectively. If idempotents of a (weakly) abundant semigroup commute then the semigroup is called (weakly) adequate. 

Our main result is that for an arbitrary band $B$, the semigroup $\ig(B)$ is  { weakly} abundant and is such that $\art$ and $\elt$ are, respectively, left and right congruences, a property called  the {\em congruence condition}. We remark that regular, abundant and restriction semigroups always have the congruence condition. On the other hand, we give an example of a band $B$ such that $\ig(B)$ is not abundant. In the positive direction we  investigate several conditions on a band $B$ that guarantee abundancy of $\ig(B)$. 

We proceed as follows. To make this article as self-contained as possible, in Section \ref{sec:pre} we recall some basics of Green's relations and regular semigroups, and of generalised Green's relations and (weakly) abundant semigroups.  We briefly describe how the presentation of any $\ig(E)$ naturally induces a reduction system. In Section \ref{sec:semilattices} we begin our investigation of free idempotent generated semigroups over bands by looking at a semilattice (that is, a commutative band) $Y.$ We show that every element of $\ig(Y)$ has a unique normal form and consequently $\ig(Y)$ has solvable word problem (a result that might be described as `folklore'). We then proceed to show that $\ig(Y)$ is abundant, and hence adequate. We remark that adequate semigroups form a quasi-variety of biunary semigroups for which the free algebras have recently been described by Kambites \cite{kambites:2011}; our semigroups $\ig(Y)$ are new and natural examples of adequate semigroups {\em not} possessing the so-called ample condition (see \cite{kambites:2011}). The adequacy of $\ig(Y)$ can be obtained as a corollary of Proposition \ref{simple normal band}, however, our straightforward early proof makes clear the strategies we subsequently use in other contexts. In a short Section \ref{sec:rectangular bands}, we show that for any rectangular band $B$, the semigroup $\ig(B)$ is regular -  another result we believe is known, but from which we want to extract specific information for our later purposes. We then proceed to look at a general band $B$ in Section \ref{sec:bands}. Unlike the case of semilattices and rectangular bands, here we may lose  uniqueness of normal forms. To overcome this problem, the concept of {\it almost normal forms} is introduced. We prove that for any band $B$, the semigroup $\ig(B)$ is  weakly abundant  with the congruence condition. We finish the section with an example of a four element non-normal band $B$ such that $\ig(B)$ is not abundant.

We then consider some sufficient conditions for $\ig(B)$ to be abundant. In Section \ref{sec:locally large} we introduce the class of {\em locally large} bands $B$, which are defined by the property that $uv=vu=v$ for all $u,v\in B$ with $BvB \subset BuB$.  We show that the word problem for $\ig(B)$ where $B$ is a locally large band is solvable. Subsequently, in Section \ref{sec:(P)}, we show that if $B$ is a locally large band or a normal band for which $\ig(B)$ satisfies a condition we label $(P),$ then $\ig(B)$ is an abundant semigroup.  We then find two classes of normal bands satisfying  Condition $(P)$. One would naturally ask here  whether $\ig(B)$ is abundant for an arbitrary normal band $B$. In Section \ref{sec:normal bands} we construct a ten element normal band $B$ with four $\mathcal{D}$-classes for which $\ig(B)$ is not abundant.

\section{Preliminaries: (weakly) abundant semigroups and Reduction systems}\label{sec:pre}

The aim of this section is to give the  technicalities needed for this article. We do not assume our readers have prior background of the area.  

Throughout this paper, for $n\in \mathbb{N}$ we write $[1,n] $ to denote $\{1, \cdots, n\}\subseteq \mathbb{N}.$ The free semigroup on a set $A$ is denoted by $A^+$; the elements of $A^+$ are words in the letters of $A$ and the binary operation is juxtaposition. The set of idempotents of  a semigroup $S$ is always denoted by $E(S)$ or more simply $E$.

We start by introducing an important tool for analysing  ideals of a semigroup $S$ and related notions of structure,
called {\it Green's relations}. There are  equivalence relations that characterise the elements of $S$
in terms of the principal ideals they generate. The two most basic of Green's relations
are $\mathcal{L}$ and $\mathcal{R}$, and are defined by
$$a~\mathcal{L}~b\Longleftrightarrow S^1a=S^1b, a~\mathcal{R}~b\Longleftrightarrow aS^1=bS^1,$$ where $S^1$ denotes $S$ with an identity element adjoined (unless $S$ already has one.) Furthermore, we denote the intersection $\mathcal{L}~\cap~\mathcal{R}$ by $\mathcal{H}$ and the join $\mathcal{L}~\vee~\mathcal{R}$  by $\mathcal{D}.$ It is known that $\mathcal{L}~\circ~\mathcal{R}=\mathcal{R}~\circ~\mathcal{L},$ and hence $\mathcal{D}=\mathcal{L}~\circ~\mathcal{R}=\mathcal{R}~\circ~\mathcal{L}.$

An element $a\in S$ is called {\it regular} if there exists $x\in S$ such that $a=axa$, that is, it is regular in the sense of von Neumann. A semigroup $S$ is {\it regular } if consists entirely of regular elements. We say that $S$ is {\it inverse} if it is regular and its idempotents commute; equivalently, they form a semilattice under the partial order $\leq$ where
$e\leq f$ if and only if $e=ef=fe$.  It is well known that $S$ is regular (inverse) if and only if each $\mathcal{L}$-class and each $\mathcal{R}$-class contain a (unique) idempotent.   Regular semigroups are particularly amenable to analysis using Green's relations.

As a generalisation of Green's relations, the  relations $\mathcal{L}^*$ and $\mathcal{R}^*$ are defined on a semigroup $S$ by the rule that
$$a~ \mathcal{L}^* ~b~\Longleftrightarrow~(\forall x,y\in S^1)~(ax=ay\Leftrightarrow bx=by)$$ and
$$a~ \mathcal{R}^* ~b~\Longleftrightarrow~(\forall x,y\in S^1)~(xa=ya\Leftrightarrow xb=yb)$$
where here $S^1$ is the convenient device of the semigroup $S$ with an identity adjoined if necessary.

It is easy to see that $\mathcal{L}\subseteq \mathcal{L}^*$, $\mathcal{R}\subseteq \mathcal{R}^*$, and if $S$ is regular, then
$\mathcal{L}=\mathcal{L}^*$ and $\mathcal{R}= \mathcal{R}^*$. We denote by $\mathcal{H}^*$ the intersection $\mathcal{L}^*~\cap~\mathcal{R}^*$, and by $\mathcal{D}^*$ the join of  $\mathcal{L}^*~\vee~\mathcal{R}^*.$ Note that unlike Green's relations, generally $\mathcal{L}^*\circ \mathcal{R}^*\neq\mathcal{R}^*\circ\mathcal{L}^*.$ 

A semigroup $S$ is {\it abundant} if each $\mathcal{L}^*$-class and each $\mathcal{R}^*$-class contains an idempotent. An abundant semigroup is
{\it adequate} if its idempotents form a semilattice. In view of the comment above, regular semigroups
are abundant while inverse semigroups are adequate. In the theory of abundant semigroups the relations $\mathcal{L}^*$, $\mathcal{R}^*$, $\mathcal{H}^*$ and $\mathcal{D}^*$  play a role which is analogous to that of
Green's relations in the theory of regular semigroups.

As an easy but useful consequence of the definition of $\mathcal{L}^*$, we have the following lemma (a dual result holds for $\mathcal{R}^*$).

\begin{lem}\cite{fountain:1982}\label{fountain}
Let $S$ be a semigroup with $a\in S$ and $e\in E(S)$. Then the following statements are equivalent:

(i) $a~\mathcal{L}^*~e$;

(ii) $ae=a$ and for any $x,y\in S^1$, $ax=ay$ implies $ex=ey$.
\end{lem}

A third set of relations, extending the starred versions of Green's relations, and useful for semigroups that are not abundant, were introduced in \cite{lawson:1991}.
 The relations $\widetilde{\mathcal{L}}$ and $\widetilde{\mathcal{R}}$ on a semigroup $S$ are defined by the rule
$$a~ \mathcal{\widetilde{L}} ~b~\Longleftrightarrow~(\forall e\in E(S))~(ae=a\Leftrightarrow be=b)$$
and $$a~ \mathcal{\widetilde{R}} ~b~\Longleftrightarrow~(\forall e\in E(S))~(ea=a\Leftrightarrow eb=b)$$
for any $a,b\in S.$

We remark here that $\mathcal{L}\subseteq \mathcal{L}^*\subseteq \mathcal{\widetilde{L}}$ and $\mathcal{R}\subseteq \mathcal{R}^*\subseteq \mathcal{\widetilde{R}}$. Moreover, if $S$ is regular, then $\mathcal{L}=\mathcal{L}^*=\mathcal{\widetilde{L}}$ and $\mathcal{R}= \mathcal{R}^*=\mathcal{\widetilde{R}}$. Whereas $\mathcal{L}^*$ and $\mathcal{R}^*$ are always right and left congruences on $S$, respectively,  the same is not necessarily true for $\mathcal{\widetilde{L}}$ and $\mathcal{\widetilde{R}}$.

A semigroup $S$ is {\it weakly abundant} if each $\mathcal{\widetilde{L}}$-class and each $\mathcal{\widetilde{R}}$-class contains an idempotent.  We say that a weakly abundant semigroup $S$ satisfies the {\it congruence condition} if $\mathcal{\widetilde{L}}$ is a right congruence and $\mathcal{\widetilde{R}}$ is a left congruence.

The following lemma is an analogue of Lemma \ref{fountain}. Of course, a dual result holds for $\widetilde{\mathcal{R}}$.

\begin{lem}\cite{lawson:1991}
Let $S$ be a semigroup with $a\in S$ and $e\in E(S)$. Then the following statements are equivalent:

(i) $a~\mathcal{\widetilde{L}}~e$;

(ii) $ae=a$ and for any $f\in E(S)$, $af=a$ implies $ef=e$.
\end{lem}

From easy observation, we have the following useful lemmas.

\begin{lem}\label{idempotents}
Let $S$ be a semigroup with $e,f \in E(S)$. Then $e~\mathcal{L}~f$ if and only if $e~\mathcal{\widetilde{L}}~f$ and $e~\mathcal{R}~f$ if and only if $e~\mathcal{\widetilde{R}}~f.$
\end{lem}

\begin{lem}\label{observation1}
Let $S$ be a semigroup, and let $a\in S$, $f\in E(S)$ be such that $a~\mathcal{\widetilde{R}}~f$ but  $a$ is not $\mathcal{R}^*$-related to $f$. Then $a$ is not $\mathcal{R}^*$-related to {\em any} idempotent of $S$.
\end{lem}

\begin{proof}
Suppose that $a~\mathcal{R}^*~e$ for some idempotent $e\in E(S)$. Then $a~\mathcal{\widetilde{R}}~e$, as $\mathcal{R}^*\subseteq \mathcal{\widetilde{R}}$, so that $e~\mathcal{\widetilde{R}}~f$  by assumption, and so $e~\mathcal{R}~f$ by Lemma \ref{idempotents}. Hence $a~\mathcal{R}^*~f$ as $\mathcal{R}\subseteq \mathcal{R}^*,$ a contradiction.
\end{proof}

\begin{lem}\label{observation2}
Let $S$ be a weakly abundant semigroup with $a\in S$ and $e\in E(S)$ such that $a~\mathcal{\widetilde{R}}~e$. Then $a~\mathcal{R}^*~e$ if and only if for any $x,y\in S$, $xa=ya$ implies that $xe=ye.$
\end{lem}
\begin{proof}
Suppose that for all $x,y\in S$, if $xa=ya$ then $xe=ye.$ By Lemma \ref{fountain}, we  need only  show that if $x\in S$ and $xa=a$, then $xe=e$. Suppose therefore that $x\in S$ and $xa=a$. As $a~\mathcal{\widetilde{R}}~e$, we have  $xa=a=ea,$ so that by assumption, $xe=ee=e$.
\end{proof}

In the rest of this section we recall the definition of reduction systems and their properties. As far as possible we follow standard notation and terminology, as may be found in \cite{Otto:1993}.

\medskip

Let $A$ be a set of objects and $\longrightarrow$  a binary relation on $A$. We call the structure $(A, \longrightarrow)$  a {\it reduction system} and the relation $\longrightarrow$ a {\it reduction relation}. The reflexive, transitive closure of $\longrightarrow$ is denoted by $\overset{*}\longrightarrow,$ while $\overset{*}\longleftrightarrow$ denotes the smallest equivalence relation on $A$ that contains $\longrightarrow.$ We denote the equivalence class of an element $x\in A$ by $[x].$  An element $x\in A$ is said to be {\it irreducible} if there is no $y\in A$ such that $x\longrightarrow y;$ otherwise, $x$ is {\it reducible.} For any $x,y\in A,$ if $ x\overset{*}\longrightarrow y$ and $y$ is irreducible, then $y$ is a {\it normal form} of $x.$ A reduction system $(A,\longrightarrow)$  is {\it noetherian} if there is no infinite sequence $x_0,x_1,\cdots~\in A$ such that for all $i\geq 0$, $x_i\longrightarrow x_{i+1}$.

We say that a reduction system $(A,\longrightarrow)$  is {\it confluent} if whenever $w,x,y\in A$, are such that $w\overset{*}{\longrightarrow}x$ and $w\overset{*}{\longrightarrow}y$, then there is a $z\in A$ such that $x\overset{*}{\longrightarrow}z$ and $y\overset{*}{\longrightarrow}z$, as described by the figure below on the left, and $(A,\longrightarrow)$   is {\it locally confluent} if whenever $w,x,y\in A$, are such that $w~{\longrightarrow}~x$ and $w~{\longrightarrow}~y$, then there is a $z\in A$ such that $x\overset{*}{\longrightarrow}z$ and $y\overset{*}{\longrightarrow}z,$  as described by the figure below on the right.
\begin{center}
\begin{tikzpicture}[scale=1.0]
\node (w1) at (0,1) {$w$};
\node (x1) at (-1,0) {$x$};
\node (y1) at (1,0) {$y$};
\node (z1) at (0,-1) {$z$};
\node (w2) at (3.5,1) {$w$};
\node (x2) at (2.5,0) {$x$};
\node (y2) at (4.5,0) {$y$};
\node (z2) at (3.5,-1) {$z$};
\path[->,font=\scriptsize,>=angle 60]
(w1) edge node[above]{*} (x1)
(w1) edge node[right]{*} (y1);
\path[->,dashed,font=\scriptsize,>=angle 60]
(x1) edge node[right]{*} (z1)
(y1) edge node[above]{*} (z1);
\path[->,font=\scriptsize,>=angle 60]
(w2) edge node[above]{} (x2)
(w2) edge node[right]{} (y2);
\path[->,dashed,font=\scriptsize,>=angle 60]
(x2) edge node[right]{*} (z2)
(y2) edge node[above]{*} (z2);
\end{tikzpicture}
\end{center}

\begin{lem}\cite{Otto:1993} \label{normal form}
Let $(A,\longrightarrow)$  be a reduction system. Then the following statements hold:

(i) If $(A,\longrightarrow)$  is noetherian and confluent, then for each $x\in A$, $[x]$ contains a unique normal form.

(ii) If $(A,\longrightarrow)$ is noetherian, then it is confluent if and only if it is locally confluent.
\end{lem}

Let $E$ be a biordered set. We use $\overline{E}^+$ to denote the free semigroup on $\overline{E}=\{\overline{e}:e\in E\}$.

\begin{lem}\label{reduction systems}
Let $E$ be a  biordered set, and let $R$ be the relation on $\overline{E}^+$ defined  by $$R=\{(\bar{e}\bar{f}, \overline{ef}):~(e,f) \mbox{~is a basic pair}\}.$$ Then $(\overline{E}^+, \longrightarrow)$ forms a noetherian reduction system, where $\longrightarrow$ is defined by
$$u\longrightarrow v  \Longleftrightarrow (\exists~(l,r)\in R)~(\exists~x,y\in \overline{E}^+) ~u=xly \mbox{~and~} v=xry.$$
\end{lem}

\begin{proof}
The proof follows directly from the definitions of the reduction system and the binary relation $\longrightarrow.$
\end{proof}

We remark here that in the reduction system $(\overline{E}^+,\longrightarrow)$ induced by $\ig(E)$, the smallest equivalence relation $\overset{*}{\longleftrightarrow}$ on $\overline{E}^+$ is exactly the congruence generated by $R.$

Finally in this section we recall that a semigroup of the form $S=X^+/\rho$, where $\rho$ is a congruence on $X^+$, has
{\em solvable word problem} if there is an algorithm to decide when two elements of $X^+$ give the same element of $S$.

\section{Free idempotent generated semigroups over semilattices}\label{sec:semilattices}

We start our investigation of free idempotent generated semigroups $\ig(B)$ over bands $B$, by looking at the special case of semilattices. Throughout this section we will use the letter $Y$ to denote a semilattice. We prove that $\ig(Y)$ is an adequate semigroup; however, it need not be regular.

It follows from Lemma \ref{reduction systems} that $\ig(Y)$  naturally induces  a noetherian  reduction system $(\overline{Y}^+, \longrightarrow)$. The next result appears to be well known to workers in this area. 

\begin{lem}\label{uniqueness}
Let $Y$ be a semilattice. Then every element in $\ig(Y)$ has a unique normal form and consequently, $\ig(Y)$ has solvable word problem.
\end{lem}

\begin{proof}
By Lemma \ref{normal form}, to show the required result we only need to argue that $(\overline{Y}^+, \longrightarrow)$ is locally confluent. For this purpose, it is sufficient to consider an arbitrary word of length 3, say $\overline{e}~\overline{f}~\overline{g}\in \overline{Y}^+$, where $e,f$ and $f,g$ are comparable. There are four cases, namely, $e\leq f\leq g$, $e\geq f\geq g,$ $e\leq f, f\geq g$ and $e\geq f, f\leq g$, for which we have the following 4 diagrams:
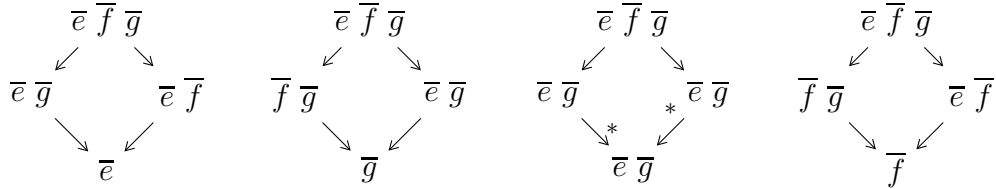
\begin{center}
\begin{tikzpicture}
\node (w1) at (0,1) {$\overline{e}~\overline{f}~\overline{g}$};
\node (x1) at (-1,0) {$\overline{e}~\overline{g}$};
\node (y1) at (1,0) {$\overline{e}~\overline{f}$};
\node (z1) at (0,-1) {$\overline{e}$};
\node (w2) at (3.5,1) {$\overline{e}~\overline{f}~\overline{g}$};
\node (x2) at (2.5,0) {$\overline{f}~\overline{g}$};
\node (y2) at (4.5,0) {$\overline{e}~\overline{g}$};
\node (z2) at (3.5,-1) {$\overline{g}$};
\node (w3) at (7,1) {$\overline{e}~\overline{f}~\overline{g}$};
\node (x3) at (6,0) {$\overline{e}~\overline{g}$};
\node (y3) at (8,0) {$\overline{e}~\overline{g}$};
\node (z3) at (7,-1) {$\overline{e}~\overline{g}$};
\node (w4) at (10.5,1) {$\overline{e}~\overline{f}~\overline{g}$};
\node (x4) at (9.5,0) {$\overline{f}~\overline{g}$};
\node (y4) at (11.5,0) {$\overline{e}~\overline{f}$};
\node (z4) at (10.5,-1) {$\overline{f}$};
\path[->,font=\scriptsize,>=angle 60]
(w1) edge node[above]{} (x1)
(w1) edge node[right]{} (y1)
(x1) edge node[right]{} (z1)
(y1) edge node[above]{} (z1)
(w2) edge node[above]{} (x2)
(w2) edge node[right]{} (y2)
(x2) edge node[right]{} (z2)
(y2) edge node[above]{} (z2)
(w3) edge node[above]{} (x3)
(w3) edge node[right]{} (y3)
(x3) edge node[right]{*} (z3)
(y3) edge node[above]{*} (z3)
(w4) edge node[above]{} (x4)
(w4) edge node[right]{} (y4)
(x4) edge node[right]{} (z4)
(y4) edge node[above]{} (z4);
\end{tikzpicture}
\end{center}

\noindent Thus $(\overline{Y}^+, \longrightarrow)$ is locally confluent, so that every element in $\ig(Y)$ has a unique normal form.

Note that an element $\overline{x_1}~\cdots~\overline{x_n}\in \ig(Y)$ is in normal form if and only if $x_i$ and $x_{i+1}$ are incomparable, for all $i\in [1, n-1].$ By uniqueness of normal forms in $\ig(Y),$  two words of $\ig(B)$ are equal if and only if the corresponding normal forms  are identical  in $\overline{E}^+$, and hence the word problem of $\ig(Y)$ is solvable.\end{proof}

\begin{pro}\label{semilattices}
The free idempotent generated semigroup $\ig(Y)$ over a semilattice $Y$ is an abundant semigroup.
\end{pro}
\begin{proof}
Let $\overline{x_{1}}\cdots~\overline{x_{n}}, \overline{y_{1}}\cdots~\overline{y_{m}}\in \ig(B)$  be in normal form. We begin with considering the product $(\overline{x_{1}}\cdots~\overline{x_{n}})(\overline{y_{1}}\cdots~\overline{y_{m}})$. Either $x_n,y_1$ are incomparable, $x_n\geq y_1$ or $x_n\leq y_1$. In the first case it is clear that $\overline{x_{1}}\cdots~\overline{x_{n}}~\overline{y_{1}}\cdots~\overline{y_{m}}$ is  a normal form. If $x_n\geq y_1$, then either $\overline{x_{1}}\cdots~\overline{x_{n-1}}~\overline{y_{1}}\cdots~\overline{y_{m}}$ is in normal form, or $y_1$ and $x_{n-1}$ are comparable. If $y_1$ and $x_{n-1}$ are comparable, then $y_1< x_{n-1}$, for we cannot have $x_{n-1}\leq y_1$ else $x_{n-1}\leq x_n$, a contradiction. Continuing in this manner we obtain  $(\overline{x_{1}}\cdots~\overline{x_{n}})(\overline{y_{1}}\cdots~\overline{y_{m}})$ has normal form $\overline{x_{1}}\cdots~\overline{x_{t-1}}~\overline{y_{1}}\cdots~\overline{y_{m}}$, where $1\leq t\leq n,$ $x_n, \cdots,x_t\geq y_1,$ and either $t=1$ (in which case $\overline{x_1}~\cdots~\overline{x_{t-1}}$ is the empty product) or $x_{t-1},y_1$ are incomparable. Similarly, if $x_n\leq y_1$, then $(\overline{x_{1}}\cdots~\overline{x_{n}})(\overline{y_{1}}\cdots~\overline{y_{m}})$ has normal form $\overline{x_{1}}\cdots~\overline{x_{n}}~\overline{y_{t+1}}\cdots~\overline{y_{m}}$, where $1\leq t\leq m,$ $x_n\leq y_1,\cdots y_t,$ and $t=m$ or $x_n, y_{t+1}$ are incomparable.

\medskip

Suppose now that $\overline{x_{1}}\cdots~\overline{x_{n}}, \overline{z_{1}}\cdots~\overline{z_{k}}$ and $\overline{y_{1}}\cdots~\overline{y_{m}}\in \ig(Y)$ are in normal form such that $$\overline{x_{1}}\cdots~\overline{x_{n}}~\overline{y_{1}}\cdots~\overline{y_{m}}=\overline{z_{1}}
\cdots~\overline{z_{k}}~\overline{y_{1}}\cdots~\overline{y_{m}}$$ 
in $\ig(Y).$ Here we assume $n,k\geq 0$ and $m\geq 1$. We proceed  to prove that \begin{equation}\label{eqn3.1}\overline{x_{1}}\cdots~\overline{x_{n}}~\overline{y_{1}}=\overline{z_{1}}
\cdots~\overline{z_{k}}~\overline{y_{1}}\end{equation} in $\ig(Y).$ If $n=k=0$ there is nothing to show. Note that the result is clearly true if $m=1$, so in what follows we assume $m\geq 2.$

\medskip

First we assume that $n\geq 1$ and $k=0$ (i.e. $\overline{z_{1}}\cdots~\overline{z_{k}}$ is empty), so that $$\overline{x_{1}}\cdots~\overline{x_{n}}~\overline{y_{1}}\cdots~\overline{y_{m}}=\overline{y_{1}}\cdots~\overline{y_{m}}.$$ In view of Lemma \ref{uniqueness}, $x_n$ and $y_1$ must be comparable. If $x_n\geq y_1$, then it follows from the above observation that $y_1\leq x_1,\cdots,x_n$, so that $\overline{x_{1}}\cdots~\overline{x_{n}}~\overline{y_{1}}=\overline{y_1}$. On the other hand, if $x_n\leq y_1$, then $$\overline{x_{1}}\cdots~\overline{x_{n}}~\overline{y_{t+1}}\cdots~\overline{y_{m}}=\overline{y_{1}}\cdots~\overline{y_{m}}$$ for $1\leq t\leq m$ such that $x_n\leq y_1,\cdots, y_t$ and $t=m$ or $x_n, y_{t+1}$ are incomparable. Then $x_n=y_t,$ so that to avoid the contradiction $y_t\leq y_{t-1}$ we  must have $t=1$. Clearly then $n=1$ and $x_1=x_n=y_1$ so that $\overline{x_1}~\overline{y_1}=\overline{y_1}$. Hence (\ref{eqn3.1}) certainly holds for $n+k+m\leq 3$.

\medskip

Suppose that $n+k+m\geq 4$ and the result is true for all $n'+k'+m'<n+k+m$. Recall we are assuming that $m\geq 2$ and in view of the above we may take $n,k
\geq 1$.

\medskip

If $x_n,y_1$ and $z_k, y_1$ are incomparable pairs, then it follows from uniqueness of normal form that $k=n$ and $\overline{x_{1}}\cdots~\overline{x_{n}}~\overline{y_{1}}=\overline{z_{1}}
\cdots~\overline{z_{k}}~\overline{y_{1}}.$

\medskip

Suppose now that $y_1\leq x_n$. Then $$\overline{x_{1}}\cdots~\overline{x_{n-1}}~\overline{y_{1}}\cdots~\overline{y_{m}}=\overline{z_{1}}
\cdots~\overline{z_{k}}~\overline{y_{1}}\cdots~\overline{y_{m}}$$ so that our induction gives us $$\overline{x_{1}}\cdots~\overline{x_{n-1}}~\overline{y_{1}}=\overline{z_{1}}
\cdots~\overline{z_{k}}~\overline{y_{1}}$$ and hence $\overline{x_{1}}\cdots~\overline{x_{n}}~\overline{y_{1}}=\overline{z_{1}}
\cdots~\overline{z_{k}}~\overline{y_{1}}.$ A similar result holds for the case $y_1\leq z_k$.

\medskip

Suppose now that $y_1\not \leq x_n$ and $y_1\not \leq z_k$ and at least one of $x_n, y_1$ or $z_k,y_1$ are comparable. Without loss of generality assume that  $x_n< y_1$. As above $x_n\leq y_1,\cdots,y_t$ for some $1\leq t\leq m$ with $t=m$ or $x_n, y_{t+1}$ incomparable. Further, there is an $r$ with $0\leq r\leq m$ such that $z_k\leq y_1,\cdots,y_r$ and $r=m$ or $z_k, y_{r+1}$ incomparable. Thus both sides of $$\overline{x_{1}}\cdots~\overline{x_{n}}~\overline{y_{t+1}}\cdots~\overline{y_{m}}=\overline{z_{1}}
\cdots~\overline{z_{k}}~\overline{y_{r+1}}\cdots~\overline{y_{m}}$$ are in normal form and so $n-t=k-r$. If $n>k$, then $r<t$, so $x_n=y_t$. To avoid the contradiction $y_t\leq y_{t-1}$, we must have $t=1$, but then $x_n=y_1$ a contradiction. Similarly, we can not have $k>n$. Hence $n=k,$ and hence $\overline{x_{1}}\cdots~\overline{x_{n}}=\overline{z_{1}}
\cdots~\overline{z_{k}},$ so that certainly $\overline{x_{1}}\cdots~\overline{x_{n}}~\overline{y_{1}}=\overline{z_{1}}
\cdots~\overline{z_{k}}~\overline{y_{1}}$  as required.
\end{proof}

We remark here that Proposition \ref{semilattices} can also be obtained as a corollary of Proposition \ref{simple normal band}, but for the sake of our readers, we have proved this special case to outline our strategy in a simple case.

\begin{coro}
The free idempotent generated semigroup $\ig(Y)$ over a semilattice $Y$ is an adequate semigroup.
\end{coro}
\begin{proof}
We have already  remarked in the beginning of Introduction that the biordered set of idempotents of $\ig(Y)$ is isomorphic to $Y$, which is a semilattice, so that $\ig(Y)$ is an adequate semigroup.
\end{proof}

\begin{ex} \cite[Example 2]{brittenham:2009} \label{ex:semilattice} {\rm Let  $Y=\{e,f,g\}$ be the semilattice with $e, f\geq g$ and $e, f$ incomparable. Then $\ig(Y)$ is not regular.}\end{ex}

\begin{proof}
  First, we observe that $$\ig(Y)=\{\overline{e},\overline{f}, \overline{g}, (\overline{e}~\overline{f})^n, (\overline{f}~\overline{e})^n, (\overline{e}~\overline{f})^n~\overline{e}, (\overline{f}~\overline{e})^n~\overline{f}:~n\in \mathbb{N}\}.$$ It is easy to check that for any $n\in \mathbb{N}$, $(\overline{e}~\overline{f})^n\in \ig(Y)$ is not regular, as for any $w\in \ig(Y),$ $(\overline{e}~\overline{f})^n w (\overline{e}~\overline{f})^n =\overline{g}$ if $w$ contains $\overline{g}$ as a letter; otherwise $(\overline{e}~\overline{f})^n w (\overline{e}~\overline{f})^n=(\overline{e}~\overline{f})^m$ for some $m\geq 2n\in \mathbb{N}.$  Therefore, $\ig(Y)$ is not a regular semigroup.\end{proof}

On the other hand, by Proposition \ref{semilattices} we have that $\ig(Y)$ is an abundant semigroup. Furthermore, $$\mathcal{R}^*=\{\{\overline{e},(\overline{e}~\overline{f})^n,(\overline{e}~\overline{f})^n~
\overline{e}:~n\in \mathbb{N} \}, \{\overline{f},(\overline{f}~\overline{e})^n,(\overline{f}~\overline{e})^n~\overline{f}:~n\in \mathbb{N} \}, \{\overline{g}\}\}$$ and $$\mathcal{L}^*=\{\{\overline{e},(\overline{f}~\overline{e})^n,(\overline{e}~\overline{f})^n~
\overline{e}:~n\in \mathbb{N} \}, \{\overline{f},(\overline{e}~\overline{f})^n,(\overline{f}~\overline{e})^n~\overline{f}:~n\in \mathbb{N} \}, \{\overline{g}\}\}$$
Note that we have $$\mathcal{D}^*=\mathcal{L}^*\circ \mathcal{R}^*=\mathcal{R}^*\circ \mathcal{L}^*$$ in $\ig(Y),$ and there are two $\mathcal{D}^*$-classes of $\ig(Y),$ namely, $\{\overline{g}\}$ and $\ig(Y)\setminus \{\overline{g}\}$,  the latter of which can be depicted by the following so called {\it egg-box} picture:

\[\begin{array}{|l|l|}
\hline 
\overline{e}, (\overline{e}~\overline{f})^n\overline{e} &  (\overline{e}~\overline{f})^n\\ \hline
 (\overline{f}~\overline{e})^n& \overline{f}, (\overline{f}~\overline{e})^n\overline{f}\\ \hline
\end{array}\]

\section{Free idempotent generated semigroups over rectangular bands}\label{sec:rectangular bands}

In this section we are concerned with the free idempotent generated semigroup $\ig(B)$ over a rectangular band $B.$ Recall from \cite{howie:1995} that a band $B$ is a semilattice $Y$ of rectangular bands $B_\alpha$, $\alpha\in Y$, where the bands $B_\alpha$ are the $\mathcal{D}=\mathcal{J}$-classes of $B$. Thus $B=\bigcup_{\alpha\in Y} B_\alpha$ where each $B_\alpha$ is a rectangular band and $B_\alpha B_\beta\subseteq B_{\alpha\beta},$ for all $ \alpha, \beta\in Y.$ At times we will use this notation without specific comment. We show that if $B$ is  a rectangular band, then $\ig(B)$ is  regular. It follows that if  $B$ is a semilattice $Y$ of rectangular bands $B_\alpha$, $\alpha\in Y,$ then any word in $\overline{B_\alpha}^+$ is regular in $\ig(B).$

\begin{lem}\label{rect-unique}
Let $B$ be a rectangular band. Then every element in $\ig(B)$ has a unique normal form.
\end{lem}

\begin{proof}
We have already remarked that the reduction system $(\overline{B}^+,\longrightarrow)$ induced by $\ig(B)$ is noetherian,  so that, according to Lemma \ref{normal form}, to demonstrate the uniqueness of normal form of elements in $\ig(B),$ we only need to prove that $(\overline{B}^+,\longrightarrow)$ is locally confluent.

For this purpose, it is sufficient to consider an arbitrary word of length 3, say $\overline{e}~\overline{f}~\overline{g}\in \overline{B}^+$, where $e,f$  and $f,g$ are comparable. Clearly, there are four cases, namely, $e~\mathcal{L}~f~\mathcal{L}~g$, $e~\mathcal{R}~f~\mathcal{R}~g,$ $e~\mathcal{L}~f~\mathcal{R}~g$ and $e~\mathcal{R}~f~\mathcal{L}~g$. Then we have the following 4 diagrams:
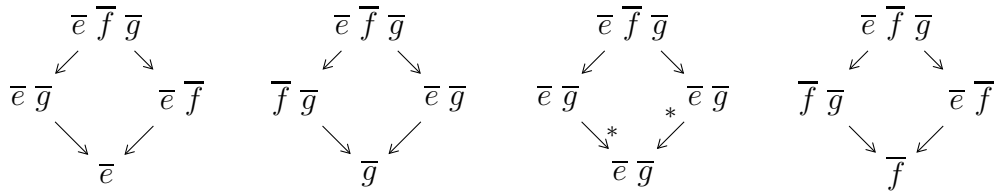
\begin{center}
\begin{tikzpicture}
\node (w1) at (0,1) {$\overline{e}~\overline{f}~\overline{g}$};
\node (x1) at (-1,0) {$\overline{e}~\overline{g}$};
\node (y1) at (1,0) {$\overline{e}~\overline{f}$};
\node (z1) at (0,-1) {$\overline{e}$};
\node (w2) at (3.5,1) {$\overline{e}~\overline{f}~\overline{g}$};
\node (x2) at (2.5,0) {$\overline{f}~\overline{g}$};
\node (y2) at (4.5,0) {$\overline{e}~\overline{g}$};
\node (z2) at (3.5,-1) {$\overline{g}$};
\node (w3) at (7,1) {$\overline{e}~\overline{f}~\overline{g}$};
\node (x3) at (6,0) {$\overline{e}~\overline{g}$};
\node (y3) at (8,0) {$\overline{e}~\overline{g}$};
\node (z3) at (7,-1) {$\overline{e}~\overline{g}$};
\node (w4) at (10.5,1) {$\overline{e}~\overline{f}~\overline{g}$};
\node (x4) at (9.5,0) {$\overline{f}~\overline{g}$};
\node (y4) at (11.5,0) {$\overline{e}~\overline{f}$};
\node (z4) at (10.5,-1) {$\overline{f}$};
\path[->,font=\scriptsize,>=angle 60]
(w1) edge node[above]{} (x1)
(w1) edge node[right]{} (y1)
(x1) edge node[right]{} (z1)
(y1) edge node[above]{} (z1)
(w2) edge node[above]{} (x2)
(w2) edge node[right]{} (y2)
(x2) edge node[right]{} (z2)
(y2) edge node[above]{} (z2)
(w3) edge node[above]{} (x3)
(w3) edge node[right]{} (y3)
(x3) edge node[right]{*} (z3)
(y3) edge node[above]{*} (z3)
(w4) edge node[above]{} (x4)
(w4) edge node[right]{} (y4)
(x4) edge node[right]{} (z4)
(y4) edge node[above]{} (z4);
\end{tikzpicture}
\end{center}

\noindent Hence $(B^*, R)$ is locally confluent.
\end{proof}

\begin{lem}\label{rectangular band}
Suppose that $B$ is a rectangular band and $\overline{u_1}~\cdots~\overline{u_n}\in \ig(B).$ Then we have $\overline{u_n}~\mathcal{L}~\overline{u_1}~\cdots~\overline{u_n}~\mathcal{R}~\overline{u_1},$ and hence $\ig(B)$ is a regular semigroup.
\end{lem}
\begin{proof}
Let $w=\overline{u_1}~\cdots~\overline{u_n}\in \ig(B)$. First we claim that $$\overline{u_1}~\cdots~\overline{u_n}~\mathcal{R}~\overline{u_1}~\cdots~\overline{u_{n-1}}.$$ Observe that $(u_n, u_{n}u_{n-1})$ and $(u_{n-1}, u_nu_{n-1})$ are both basic pairs. Hence we have
$$
\begin{aligned}
\overline{u_1}~\cdots~\overline{u_{n-1}}~\overline{u_n}~\overline{u_nu_{n-1}}&=\overline{u_1}~\cdots~
\overline{u_{n-1}}~\overline{u_nu_nu_{n-1}} \ \ \ \ \\
                                              & =
\overline{u_1}~\cdots~\overline{u_{n-1}}~\overline{u_nu_{n-1}}\ \ \ \ \\
                                              &=
\overline{u_1}~\cdots~\overline{u_{n-1}u_nu_{n-1}} \\
                                              &=
\overline{u_1}~\cdots~\overline{u_{n-1}},
\end{aligned}
$$
so that $\overline{u_1}~\cdots~\overline{u_n}~\mathcal{R}~\overline{u_1}~\cdots~\overline{u_{n-1}}.$ By finite induction we obtain that $\overline{u_1}~\cdots~\overline{u_{n}}~\mathcal{R}~\overline{u_1}.$

\medskip

Similarly, we can show that  $\overline{u_1}~\overline{u_2}~\cdots~\overline{u_n}~\mathcal{L}~\overline{u_n}.$ Certainly then $\ig(B)$ is regular.  \end{proof}

\begin{coro}\label{regular of IG(B)}
Let $B$ be a semilattice $Y$ of rectangular bands $B_\alpha,$ $\alpha\in Y.$ Then for any $x_1, \cdots, x_n\in B_\alpha,$  $\overline{x_1}~\cdots~\overline{x_n}$ is a regular element of $\ig(B).$
\end{coro}
\begin{proof}
It is clear from the presentations of $\ig(B_\alpha)$ and $\ig(B)$ that there is a well defined morphism $$\overline{\psi}: \ig(B_\alpha)\longrightarrow \ig(B), \mbox{~such~that~}\overline{e}~\overline{\psi}=\overline{e}$$ for each $e\in B_\alpha$. It follows from Lemma \ref{rectangular band} that for any $x_1, \cdots, x_n\in B_\alpha,$  $\overline{x_1}~\cdots~\overline{x_n}$ is regular in $\ig(B_\alpha).$ Since clearly $\overline{\psi}$ preserves  regularity, we have that $(\overline{x_1}~\cdots~\overline{x_n})~\overline{\psi}=\overline{x_1}~\cdots~\overline{x_n}$ is regular in $\ig(B).$
\end{proof}
\section{Free idempotent generated semigroups over bands}\label{sec:bands}

Our aim here is to investigate the general structure of $\ig(B)$ for an arbitrary band $B$. We  prove that for any band $B$, the semigroup $\ig(B)$ is  weakly abundant  with the congruence condition. However, we demonstrate a band $B$ for which $\ig(B)$ is not abundant.

\begin{lem}\label{homo}
Let $S$ and $T$ be semigroups with biordered sets of idempotents $U=E(S)$ and $V=E(T)$, respectively, and let $\theta: S\longrightarrow T$ be a morphism. Then the map from $\overline{U}$ to $\overline{V}$ defined by $\overline{e}\mapsto \overline{e\theta}$, for all $e\in U$, lifts to a well defined morphism $\overline{\theta}: \ig(U)\longrightarrow \ig(V)$.
\end{lem}

\begin{proof}
Since $\theta$ is a morphism by assumption, we have that $(e,f)$ is a basic pair in $U$ implies $(e\theta,f\theta)$ is a basic pair in $V$, so that there exists a morphism $\overline{\theta}: \ig(U)\longrightarrow \ig(V)$ defined by $\overline{e}~\overline{\theta}=\overline{e\theta}$, for all $e\in U$.
\end{proof}

Let $B$ be a band, which for the rest of this section we write as  a semilattice $Y$ of rectangular bands $B_\alpha$, $\alpha\in Y$. The mapping $\theta$ defined by $$\theta: B\longrightarrow Y, x\mapsto \alpha$$ where $x\in B_{\alpha}$, is a morphism with kernel $\mathcal{D}$. Applying Lemma \ref{homo} to this $\theta$,  we have the following corollary.

\begin{coro}\label{direct}
Let $B=\bigcup_{\alpha\in Y} B_\alpha$ be a semilattice $Y$ of rectangular bands $B_\alpha$, $\alpha\in Y$. Then a map $\overline{\theta}: \ig(B)\longrightarrow \ig(Y)$ defined by $$(\overline{x_1}~\cdots~\overline{x_n})~\overline{\theta}=\overline{\alpha_1}~\cdots~\overline{\alpha_n}$$ is a morphism, where $x_i\in B_{\alpha_i}$, for all $i\in [1,n]$.
\end{coro}

To proceed further we need the following definition of {\it left to right significant indices} of elements in $\ig(B)$.

\medskip

Let $\overline{x_1}\cdots~\overline{x_n}\in \overline{B}^+$ with $x_i\in B_{\alpha_i}$, for all $1\leq i\leq n$. Then a set of numbers $$\{i_1,\cdots,i_r\}\subseteq [1, n] \mbox{~with~} i_1< \cdots<i_r$$ is called the {\it left to right significant indices} of $\overline{x_1}\cdots~\overline{x_n}$, if these numbers are picked out in the following manner:

$i_1:$  the largest number such that $\alpha_1, \cdots, \alpha_{i_1}\geq \alpha_{i_1}$;

$k_1:$ the largest number such that $\alpha_{i_1}\leq \alpha_{i_1},\alpha_{i_1+1},\cdots,\alpha_{k_1}.$
\medskip

We pause here to remark that $\alpha_{i_1}, \alpha_{k_1+1}$ are incomparable. Because, if $\alpha_{i_1}\leq \alpha_{k_1+1}$, we add $1$ to $k_1$, contradicting the choice of $k_1;$ and if $\alpha_{i_1}>\alpha_{k_1+1}$, then $\alpha_1,\cdots,\alpha_{i_1},\cdots,\alpha_{k_1}\geq \alpha_{k_1+1}$, contradicting the choice of $i_1$. Now we continue our process:

\medskip

$i_2:$ the largest number such that $\alpha_{k_1+1}, \cdots,\alpha_{i_2}\geq \alpha_{i_2};$

$k_2:$ the largest number such that $\alpha_{i_2}\leq \alpha_{i_2}, \alpha_{i_2+1}, \cdots, \alpha_{k_2}.$

$\vdots$

$i_r:$ the largest number such that $\alpha_{k_{r-1}+1}, \cdots, \alpha_{i_r}\geq \alpha_{i_r};$

$k_r=n$: here we have $\alpha_{i_r}\leq \alpha_{i_r}, \alpha_{{i_r}+1}, \cdots, \alpha_n.$ Of course, we may have $i_r=k_r=n.$

\medskip

Corresponding to the so called left to right significant indices $i_1, \cdots, i_r,$ we have $$\alpha_{i_1},\cdots, \alpha_{i_r}\in Y.$$ We claim that for all $1\leq s\leq r-1$, $\alpha_{i_s}$ and $\alpha_{i_{s+1}}$ are incomparable. If not, suppose that there exists some $1\leq s\leq r-1$ such that $\alpha_{i_s}\leq \alpha_{i_{s+1}}.$ Then $\alpha_{i_s}\leq \alpha_{k_s+1}$ as $\alpha_{i_{s+1}}\leq \alpha_{k_s+1},$ a contradiction; if $\alpha_{i_s}\geq \alpha_{i_{s+1}},$ then $\alpha_{i_{s+1}}\leq \alpha_{i_{s+1}}, \alpha_{i_{s+1}-1},\cdots, \alpha_{k_{s-1}+1}$ with $k_0=0$, contradicting our choice of $i_s.$ Therefore,  we deduce that $\overline{\alpha_{i_1}}~\cdots~\overline{\alpha_{i_r}}$ is the unique normal form of $\overline{\alpha_1}~\cdots~\overline{\alpha_n}$ in $\ig(Y).$

We can use the following {\it Hasse diagram} to depict the relationship among $\alpha_1, \cdots, \alpha_{i_r}:$

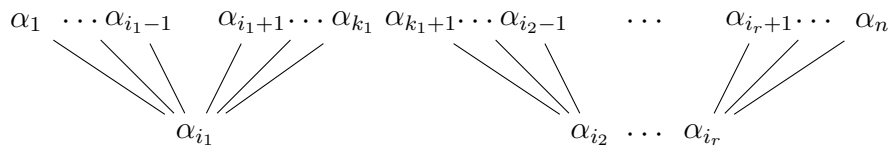
\begin{center}
\begin{tikzpicture}[scale=1.5]
\node (A) at (0,1) {$\alpha_1$};
\node (B) at (0.5,1) {$\cdots$};
\node (C) at (1,1) {$\alpha_{i_1-1}$};
\node (D) at (1.5,0) {$\alpha_{i_1}$};
\node (E) at (2,1) {$\alpha_{i_1+1}$};
\node (F) at (2.5,1) {$\cdots$};
\node (G) at (2.9,1) {$\alpha_{k_1}$};
\node (H) at (3.5,1) {$\alpha_{k_1+1}$};
\node (I) at (4,1) {$\cdots$};
\node (J) at (4.5,1) {$\alpha_{i_2-1}$};
\node (K) at (5,0) {$\alpha_{i_2}$};
\node (M) at (5.5,0) {$\cdots$};
\node (N) at (6,0) {$\alpha_{i_r}$};
\node (L) at (5.5,1) {$\cdots$};
\node (P) at (6.5,1) {$\alpha_{i_r+1}$};
\node (Q) at (7,1) {$\cdots$};
\node (R) at (7.5,1) {$\alpha_{n}$};
\path[-,font=\scriptsize,>=angle 60]
(A) edge node[above]{} (D)
(B) edge node[above]{} (D)
(C) edge node[above]{} (D)
(E) edge node[above]{} (D)
(F) edge node[above]{} (D)
(G) edge node[above]{} (D)
(H) edge node[above]{} (K)
(I) edge node[above]{} (K)
(J) edge node[above]{} (K)
(P) edge node[above]{} (N)
(Q) edge node[above]{} (N)
(R) edge node[above]{} (N);

\end{tikzpicture}
\end{center}

Dually, we can define the  {\it right to left significant indices} $\{l_1, \cdots, l_s\}\subseteq [1,n]$ of the element $\overline{x_1}\cdots~\overline{x_n}\in \overline{B}^+$, where $l_1<\cdots<l_s.$ Note that as $\overline{\alpha_{i_1}}~\cdots~\overline{\alpha_{i_r}}$ must equal to  $\overline{\alpha_{l_1}}~\cdots~\overline{\alpha_{l_s}}$ in $\overline{B}^+$, we have $r=s.$

\begin{lem}\label{indices are equal}
Let $\overline{x_1}\cdots~\overline{x_n}\in \overline{B}^+$ with $x_i\in \alpha_{i}$, for all $i\in [1,n]$, and left to right significant indices $i_1,\cdots, i_r$. Suppose also that  $\overline{y_1}\cdots~\overline{y_m}\in \overline{B}^+$ with $y_i\in \beta_{i}$, for all $i\in[1,m]$, and left to right significant indices $l_1,\cdots, l_s$. Then $$\overline{x_1}\cdots~\overline{x_n}=\overline{y_1}\cdots~\overline{y_m}$$ in $\ig(B)$ implies $s=r$ and $\alpha_{i_1}=\beta_{l_1}, \cdots, \alpha_{i_r}=\beta_{l_r}.$
\end{lem}

\begin{proof}
It follows from Lemma \ref{direct} and the discussion above that $$\overline{\alpha_{i_1}}~\cdots~\overline{\alpha_{i_r}}=\overline{\alpha_1}~\cdots~\overline{\alpha_n}=
\overline{\beta_1}~\cdots~\overline{\beta_m}=\overline{\beta_{l_1}}~\cdots~\overline{\beta_{l_s}}$$ in $\ig(Y).$ By uniqueness of normal form, we have that $s=r$ and $\alpha_{i_1}=\beta_{l_1}, \cdots, \alpha_{i_r}=\beta_{l_r}.$
\end{proof}

In view of  the above observations, we introduce the following notions.

\medskip

Let $w=\overline{x_1}~\cdots\overline{x_n}$ be a word in $\overline{B}^+$ with $x_i\in B_{\alpha_i}$, for all $i\in [1,n].$ Suppose that $w$ has left to right significant indices $i_1, \cdots, i_r.$ Then we call the natural number $r$ the {\it $Y$-length}, and $\alpha_{i_1}, \cdots, \alpha_{i_r}$ the {\it ordered $Y$-components} of the equivalence class of $w$ in $\ig(B).$

In what follows whenever we write $w\sim w'$ for $w, w'\in \overline{B}^+$, we mean that the word $w'$ can be obtained from the word $w$  from a single step $\longrightarrow$ or its reverse $\longleftarrow$ as in Lemma~\ref{reduction systems}. 

\begin{lem}\label{useful}
Let  $\overline{x_1}~\cdots~\overline{x_n}\in \overline{B}^+$ with left to right significant indices $i_1, \cdots, i_r,$ where $x_i\in B_{\alpha_i},$ for all $i\in [1,n].$ Let $\overline{y_1}\cdots~\overline{y_m}\in \overline{B}^+$ be 
such that  $\overline{y_1}\cdots~\overline{y_m}\sim \overline{x_1}~\cdots~\overline{x_n}$, and suppose that the left to right significant indices of $\overline{y_1}\cdots~\overline{y_m}$ are  $j_1,\cdots,j_r$. Then for all $l\in [1,r],$ we have $$\overline{y_1}~\cdots~\overline{y_{j_l}}=
\overline{x_1}~~\cdots~\overline{x_{i_l}}~\overline{u}$$ and $y_{j_l}=u'x_{i_l}u,$ where $u'=\varepsilon$ or $u'\in B_\sigma$ with $\sigma\geq \alpha_{i_l},$ and either $u=\varepsilon$,  or $u\in B_\delta$ for some $\delta> \alpha_{i_l}$, or $u\in B_{\alpha_{i_l}}$ and there exists $v\in B_{\theta}$ with $\theta>\alpha_{i_l}$, $vu=u$ and $uv=x_{i_l}$.
\end{lem}
\begin{proof}
Suppose that we split $x_k=uv$ for some $k\in [1,n],$ where $uv$ is a basic product with $u\in B_\mu$ and $v\in B_\tau,$ so that $\alpha_k=\mu \tau.$ Then
$$\overline{x_1}~\cdots~\overline{x_n}\sim \overline{x_1}~\cdots~\overline{x_{k-1}}~\overline{u}~\overline{v}~\overline{x_{k+1}}~\cdots~\overline{x_n}=\overline{y_1}~\cdots~\overline{y_m}.$$

If $k<i_l$,  then clearly $y_{j_l}=x_{i_l}$ and
$$\overline{y_1}~\cdots~\overline{y_{j_l}}
=\overline{x_1}~\cdots~\overline{x_{k-1}}~\overline{u}~\overline{v}~\overline{x_{k+1}}\cdots~
\overline{x_{i_l}}=\overline{x_1}~~\cdots~\overline{x_{i_l}},$$ so we may take $u=u'=\varepsilon.$

If $k=i_l,$ then $\mu\tau=\alpha_{i_l}.$ If $\mu\geq \tau$, then $y_{j_l}=v$ and again  $$\overline{y_1}~\cdots~\overline{y_{j_l}}=\overline{x_1}~\cdots \overline{x_{i_l-1}}~\overline{u}~\overline{v}=\overline{x_1}~\cdots~\overline{x_{i_l}}.$$ As $x_{i_l}=uv~\mathcal{L}~v$, we have $y_{j_l}=v=vx_{i_l}$. Also, $x_{i_l}=uv=uy_{j_l}.$

\medskip

On the other hand, if $\mu<\tau$, then $y_{j_l}=u$. As $uv$ is a basic product, $uv=u=x_{i_l}$ or $vu=u$. If $uv=u=x_{i_l}$, then  $$\overline{y_1}\cdots \overline{y_{j_1}}=\overline{x_1}\cdots \overline{x_{i_l-1}}~\overline{u}=\overline{x_1}\cdots \overline{x_{i_l}},$$ and $y_{j_l}=u=uv=x_{i_l}.$ If $vu=u$, then as $x_k=uv~\mathcal{R}~u$ and $u=uvu$,  $$\overline{y_1}\cdots \overline{y_{j_1}}=\overline{x_1}\cdots~\overline{x_{i_l-1}}~\overline{u}=\overline{x_1}\cdots~\overline{x_{k-1}}~
\overline{uv}~\overline{u}=\overline{x_1}\cdots~\overline{x_{i_l}}~\overline{u}$$ and $y_{j_l}=x_{i_l}u$ where $vu=u.$
Also,
$$\overline{x_1}\cdots~\overline{x_{i_l}}=\overline{x_1}\cdots~\overline{x_{k-1}}~
\overline{uv}=\overline{x_1}~\cdots~\overline{x_{i_l}}~\overline{u}~\overline{v}=\overline{y_1}~\cdots~\overline{y_{j_l}}~\overline{v}$$ and $x_{i_l}=y_{j_l}v.$

\medskip

Finally, suppose that $k>i_l$. Then it is obviously that $j_l=i_l,$ $x_{i_l}=y_{j_l}$ and $$\overline{y_1}~\cdots~\overline{y_{j_l}}
=\overline{x_1}~\cdots~\overline{x_{i_l}}.$$
\end{proof}

It follows immediately from Lemma \ref{useful} that
\begin{coro}\label{important}
Suppose that $\overline{y_1}\cdots~\overline{y_m}=\overline{x_1}~\cdots~\overline{x_n}\in \ig(B)$ with left to right significant indices $j_1,\cdots,j_r$ and $i_1, \cdots, i_r$, respectively, and suppose $x_i\in B_{\alpha_i}$ for all $i\in [1,n].$ Then for all $l\in [1,r]$, we have $$\overline{y_1}~\cdots~\overline{y_{i_l}}=
\overline{x_1}~\cdots~\overline{x_{i_l}}~\overline{u_1}~\overline{u_2}\cdots~\overline{u_s}$$ and $y_{j_l}=u_s'\cdots u_1'x_{i_l}u_1\cdots u_s,$ where for all $t\in [1,s],$ $u_t'=\varepsilon$ or $u_t'\in B_{\sigma_t}$ for some $\sigma_t\geq \alpha_{i_l}$, and either $u_t=\varepsilon$ or $u_t\in B_{\delta_t}$ for some $\delta_t> \alpha_{i_l}$, or $u_t\in B_{\alpha_{i_l}}$ and there exists $v_t\in B_{\theta_t}$ with $\theta_t>\alpha_{i_l}$ and  $v_t u_t=u_t$. Consequently, $\overline{y_1}~\cdots~\overline{y_{j_l}}~\mathcal{R}~
\overline{x_1}~\cdots~\overline{x_{i_l}},$ and hence $y_1~\cdots~y_{j_l}~\mathcal{R}~
x_1~\cdots~x_{i_l}.$
\end{coro}

\begin{proof}
The proof follows from Lemma \ref{useful} by finite induction.
\end{proof}

Note that the duals of Lemma \ref{useful} and Corollary \ref{important} hold for right to left significant indices.

\medskip

From  Lemmas \ref{uniqueness} and \ref{rect-unique}, we know that if $B$ is a semilattice or a rectangular band, then every element in $\ig(B)$ has a unique normal form. However, it may not true for an arbitrary band $B$, even if $B$ is normal. Recall that a band $B$ is {\em normal} if it satisfies the identity $xyzx=xzyx$. Equivalently, $B$ is a {\em strong} semilattice of rectangular bands, that is, $$B=\mathcal{B}(Y;B_{\alpha}, \phi_{\alpha, \beta})$$ is a semilattice $Y$ of rectangular bands $B_\alpha$, $\alpha\in Y,$ such that for all $\alpha\geq\beta$ in $Y$ there exists a morphism $\phi_{\alpha,\beta}: B_{\alpha}\longrightarrow B_{\beta}$ such that

\medskip

(B1) for all $\alpha\in Y$, $\phi_{\alpha,\alpha}=1_{B_{\alpha}};$

\medskip

(B2) for all $\alpha, \beta, \gamma\in Y$ such that $\alpha\geq\beta\geq\gamma,$ $\phi_{\alpha,\beta}\phi_{\beta,\gamma}=\phi_{\alpha, \gamma},$

\medskip
\noindent and for all $\alpha,\beta\in Y$ and $x\in B_{\alpha}, y\in B_\beta$, $$xy=(x\phi_{\alpha,\alpha\beta})(y\phi_{\beta,\alpha\beta}).$$

\begin{ex} {\rm Let $B=\mathcal{B}(Y; B_{\alpha}, \phi_{\alpha, \beta})$ be a strong semilattice $Y=\{\alpha,\beta,\gamma,\delta\}$ of rectangular bands $B_\alpha$, $\alpha\in Y$ (see the figure below), such that $\phi_{\alpha, \beta}$ is defined by $a \phi_{\alpha,\beta}=b,$ the remaining morphisms being defined in the obvious unique manner.

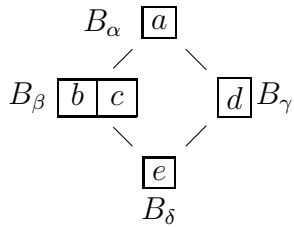
\begin{center}
\begin{tikzpicture}[scale=0.5]
\node (a) at (-1.5,2.0) {$B_\alpha$};
\node (w) at (0,2) {\boxed{a}};
\node (b) at (-3.5,0) {$B_\beta$};
\node (x) at (-2,0) {\begin{tabular}{ r|c|c| }
 \cline{2-3}
 & $b$ & $c$ \\
 \cline{2-3}
 \end{tabular}};
\node (y) at (2,0) {\boxed{d}};
\node (c) at (3.0,0) { $\, B_\gamma$};
\node (z) at (0,-2) {\boxed{e}};
\node (a) at (0,-3.0) {$B_\delta$};
\path[-,font=\scriptsize,>=angle 60]
(w) edge node[above]{} (x)
(w) edge node[right]{} (y)
(x) edge node[right]{} (z)
(y) edge node[above]{} (z);
\end{tikzpicture}
\end{center}

By an easy calculation, we have $$\overline{c}~\overline{d}=\overline{c}~\overline{ad}=\overline{c}~\overline{a}~\overline{d}=\overline{ca}~\overline{d}=\overline{b}~\overline{d}$$
in $\ig(B)$, so that not every element in $\ig(B)$ has a unique normal form.}
\end{ex}

\begin{lem}\label{almost normal form}
Let $B=\bigcup_{\alpha\in Y} B_{\alpha}$ be a semilattice $Y$ of rectangular bands $B_\alpha, \alpha\in Y$. Let $ \overline{x_1}~\cdots~\overline{x_n}\in \ig(B)$ with $x_{i}\in B_{\alpha_i}$, for all $i\in [1,n]$, and let $y\in B_\beta$ with $\beta\leq \alpha_i$, for all $i\in [1,n]$. Then in $\ig(B)$ we have $$\overline{x_1}~\cdots~\overline{x_n}~\overline{y}=\overline{x_1\cdots x_nyx_n\cdots x_1}~\cdots~\overline{x_{n-1}x_nyx_nx_{n-1}}~\overline{x_nyx_n}~\overline{y}$$ and
$$\overline{y}~\overline{x_1}~\cdots~\overline{x_n}=\overline{y}~\overline{x_1yx_1}~\overline{x_2x_1yx_1x_2}~\cdots~\overline{x_n\cdots x_1yx_1\cdots x_n}.$$
\end{lem}

\begin{proof}
First, we notice that for any $x\in B_\alpha, y\in B_\beta$ such that $\alpha\geq \beta$, we have $yx~\mathcal{R}~ y$, so that $(y,yx)$ is a basic pair and $(yx)y=y$. On the other hand, as $(yx)x=yx$, we have that $(x,yx)$ is a basic pair, so that $$\overline{x}~\overline{y}=\overline{x}~\overline{(yx)y}=\overline{x}~\overline{yx}~\overline{y}=\overline{xyx}~\overline{y}.$$
Thus, the first required equality follows from the above observation by finite induction. Dually, we can show the second one.
\end{proof}

\begin{coro}\label{almost normal form-simple}
Let $B=\mathcal{B}(Y; B_\alpha, \phi_{\alpha,\beta})$ be a normal band and let $\overline{x_1}~\cdots~\overline{x_n}\in \ig(B)$ be such that $x_{i}\in B_{\alpha_i}$, for all $i\in [1,n]$. Let $y\in B_\beta$ with $\beta\leq \alpha_i$, for all $i\in [1,n]$. Then in $\ig(B)$ we have $$\overline{x_1}~\cdots~\overline{x_n}~\overline{y}=\overline{x_1\phi_{\alpha_1,\beta}}~\cdots~\overline{x_n\phi_{\alpha_n,\beta}}~\overline{y}$$ and
$$\overline{y}~\overline{x_1}~\cdots~\overline{x_n}=\overline{y}~\overline{x_1\phi_{\alpha_1,\beta}}~\cdots~\overline{x_n\phi_{\alpha_n,\beta}}.$$
\end{coro}

\begin{coro}
Let $B=\bigcup_{\alpha\in Y} B_{\alpha}$ be a chain $Y$ of rectangular bands $B_\alpha$, $\alpha\in Y$. Then $\ig(B)$ is a regular semigroup.
\end{coro}

\begin{proof}
Let $\overline{u_1}~\cdots~\overline{u_n}$ be an element in $\ig(B)$. From Lemma \ref{almost normal form} it follows that $\overline{u_1}~\cdots~\overline{u_n}$  can be written as an element of $\ig(B)$ in which all letters come from $B_\gamma$, where $\gamma$ is the minimum of the ordered $Y$-components $\{\alpha_1, \cdots, \alpha_n\}$, so that $\overline{u_1}~\cdots~\overline{u_n}$ is regular  by Lemma \ref{regular of IG(B)}.
\end{proof}

Given the above observations, we now introduce the idea of {\it almost normal form} for elements in $\ig(B)$.

\begin{definition}

An element $\overline{x_1}~\cdots~\overline{x_n}\in \overline{B}^+$ is said to be in {\it almost normal form} if there exists a sequence $$1\leq i_1<i_2<\cdots<i_{r-1}\leq n$$ with $$\{x_1, \cdots, x_{i_1}\}\subseteq B_{\alpha_1}, \{x_{i_1+1}, \cdots, x_{i_2}\}\in B_{\alpha_2}, \cdots, \{x_{i_{r-1}+1},\cdots x_{n}\}\subseteq B_{\alpha_r}$$ where $\alpha_i, \alpha_{i+1}$ are incomparable for all $i\in[1,r-1].$

\end{definition}

It is obvious that the element $\overline{x_1}~\cdots~\overline{x_{n}}\in \overline{B}^+$  above has 
$Y$-length $r$, ordered $Y$-components $\alpha_1, \cdots, \alpha_r$, left to right significant indices $i_1, i_2, \cdots, i_{r-1}, i_r=n$ and right to left significant indices $1, i_{1}+1, \cdots,  i_{r-2}+1, i_{r-1}+1$.  Note that, in general, the almost normal forms of elements of $\ig(B)$ are not unique. Further, if $\overline{x_1}~\cdots~\overline{x_{n}}=\overline{y_1}~\cdots~\overline{y_{m}}$ are in almost normal form, then they have the same $Y$-length and ordered $Y$-components, but the  significant indices of the expressions on each side can  differ.

\medskip

The next result is immediate from the definition of significant indices and Lemma \ref{almost normal form}.

\begin{lem}
Every element of $\ig(B)$ can be written in almost normal form.
\end{lem}

We have the following lemma regarding the almost normal form of the product of two almost normal forms.

\begin{lem}\label{product}
Let $\overline{x_1}~\cdots~\overline{x_n}\in \ig(B)$ be in almost normal form with $Y$-length $r$, left to right significant indices $i_1, \cdots, i_r=n$ and ordered $Y$-components $\alpha_1, \cdots, \alpha_r$, let $\overline{y_1}~\cdots~\overline{y_m}\in \ig(B)$ be in almost normal form with  $Y$-length $s$, left to right significant indices $l_1, \cdots, l_s=m$ and ordered $Y$-components $\beta_1, \cdots, \beta_s$. Then $($with $i_0=0$$)$

\medskip

(i) $\alpha_r$ and $\beta_1$ incomparable implies that $\overline{x_1}~\cdots~\overline{x_{i_r}}~\overline{y_1}~\cdots~\overline{y_{l_s}}$ is in almost normal form;

\medskip

(ii) $\alpha_r\geq \beta_1$ implies $$\overline{x_1}\cdots \overline{x_{i_t}}~\overline{x_{i_t+1}\cdots x_{i_r}y_1x_{i_r}\cdots x_{i_t+1}}~\cdots~\overline{x_{i_r}y_1x_{i_r}}~\overline{y_1}~\cdots~\overline{y_{l_s}}$$ is an almost normal form of
 the product $\overline{x_1}~\cdots~\overline{x_{i_r}}~\overline{y_1}~\cdots~\overline{y_{l_s}},$ for some $t\in [0,r-1]$ such that $\alpha_r, \cdots, \alpha_{t+1}\geq \beta_1$ and $t=0$ or $\alpha_t, \beta_1$ are incomparable;

\medskip

(iii) $\alpha_r\leq \beta_1$ implies $$\overline{x_1}~\cdots~\overline{x_{i_r}}~\overline{y_1x_{i_r}y_1}~\cdots~\overline{y_{l_v}\cdots y_1x_{i_r}y_1\cdots y_{l_v}}~\overline{y_{l_v+1}}~\cdots~\overline{y_{l_s}}$$ is an almost normal form of  the product $\overline{x_1}~\cdots~\overline{x_{i_r}}~\overline{y_1}~\cdots~\overline{y_{l_s}}$ for some $v\in [1,s]$ such that $\alpha_r\leq\beta_1, \cdots, \beta_v$ and $v=s$ or $\beta_{v+1},\alpha_r$ are incomparable.

\end{lem}

\begin{proof}
Clearly, the statement (i) is true. We now aim to show (ii). Since $\alpha_r\geq\beta_1$, we have $$\overline{x_{i_{r-1}+1}}~\cdots~\overline{x_{i_r}}~\overline{y_1}=\overline{x_{i_{r-1}+1}\cdots x_{i_r}y_1x_{i_r}\cdots x_{i_{r-1}+1}}~\cdots~\overline{x_{i_r}y_1x_{i_r}}~\overline{y_1}$$ by Corollary \ref{almost normal form}. Consider $\alpha_{r-1}$ and $\beta_1$, then we either have $\alpha_{r-1}\geq\beta_1$ or they are incomparable, as $\alpha_{r-1}< \beta_1$ would imply $\alpha_r> \alpha_{r-1}$, which contradicts the almost normal form of $\overline{x_1}~\cdots~\overline{x_{i_r}}$. By finite induction we have that $$\overline{x_1}\cdots \overline{x_{i_t}}~\overline{x_{i_t+1}\cdots x_{i_r}y_1x_{i_r}\cdots x_{i_t+1}}~\cdots~\overline{x_{i_r}y_1x_{i_r}}~\overline{y_1}~\cdots~\overline{y_{l_s}}$$ is an almost normal form of the product $\overline{x_1}~\cdots~\overline{x_{i_r}}~\overline{y_1}~\cdots~\overline{y_{l_s}},$ for some $t\in [0,r-1],$ such that $\alpha_r, \cdots, \alpha_{t+1}\geq \beta_1$ and $t=0$ or $\alpha_t, \beta_1$ are incomparable. Similarly, we can show (iii).
\end{proof}

\begin{them}\label{IG(band)}
Let  $B$ be a semilattice $Y$ of rectangular bands $B_\alpha,\alpha\in Y.$ Then $\ig(B)$ is a weakly abundant semigroup with the congruence condition.
\end{them}

\begin{proof}
Let $\overline{x_1}~\cdots~\overline{x_n}\in \ig(B)$ be in almost normal form with $Y$-length $r$, left to right significant indices $i_1, \cdots, i_r=n$, and ordered $Y$-components $\alpha_1, \cdots, \alpha_r$. Clearly $\overline{x_1}~\overline{x_1}~\cdots~\overline{x_n}=\overline{x_1}~\cdots~\overline{x_n}.$ Let $e\in B_\delta$ be such that $\overline{e}~\overline{x_1}~\cdots~\overline{x_n}=\overline{x_1}~\cdots~\overline{x_n}.$ By Corollary \ref{direct},  applying $\overline{\theta}$, we have that $\overline{\delta}~\overline{\alpha_1}~\cdots~\overline{\alpha_r}=\overline{\alpha_1}~\cdots~\overline{\alpha_r}.$
It follows from Lemma \ref{uniqueness} that $\delta\geq \alpha_1$, so that by Corollary \ref{important} we have $$ex_1\cdots x_{i_1}~\mathcal{R}~x_1\cdots x_{i_1}.$$ On the other hand, $x_1\cdots x_{i_1}~\mathcal{R}~x_1$  so that $ex_1~\mathcal{R}~x_1$, and we have $x_1\leq_{\mathcal{R}} e$. Thus  $\overline{e}~\overline{x_1}=\overline{ex_1}=\overline{x_1}.$ Therefore $\overline{x_1}~\cdots~\overline{x_n}~\mathcal{\widetilde{R}}~\overline{x_1}.$
Dually, $\overline{x_1}~\cdots~\overline{x_n}~\mathcal{\widetilde{L}}~\overline{x_n},$ so that $\ig(B)$ is a weakly abundant semigroup as required.

\medskip

Next we show that $\ig(B)$ satisfies the congruence condition.

\medskip

Let $\overline{x_1}~\cdots~\overline{x_n}\in \ig(B)$ be defined as above and let $\overline{y_1}~\cdots~\overline{y_m}\in \ig(B)$ be in almost normal form with $Y$-length $u$, left to right significant indices $l_1, \cdots, l_u=m$ and ordered $Y$-components $\beta_1, \cdots, \beta_u.$ From the above and a comment in Section \ref{sec:intro}, we have $\overline{x_1}~\cdots~\overline{x_n}~\mathcal{\widetilde{R}}~\overline{y_1}~\cdots~\overline{y_m}$ if and only if $x_1~\mathcal{R}~y_1.$ Suppose now that $x_1~\mathcal{R}~y_1,$  so that $\alpha_1=\beta_1.$ Let
$\overline{z_1}~\cdots~\overline{z_s}\in \ig(B)$, where, without loss of generality, we can assume it is in almost normal form with $Y$-length $t,$ left to right significant indices $j_1, \cdots, j_t=s$, and ordered $Y$-components $\gamma_1, \cdots, \gamma_t.$  We aim to show that $$\overline{z_1}~\cdots~\overline{z_s}~\overline{x_1}~\cdots~\overline{x_n}~\mathcal{\widetilde{R}}~ \overline{z_1}~\cdots~\overline{z_s}~\overline{y_1}~\cdots~\overline{y_m}.$$

We consider the following three cases.

\medskip

(i) If $\alpha_1=\beta_1, \gamma_t$ are incomparable, then it is clear that   $$\overline{z_1}~\cdots~\overline{z_s}~\overline{x_1}~\cdots~\overline{x_n}\mbox{~and~} \overline{z_1}~\cdots~\overline{z_s}~\overline{y_1}~\cdots~\overline{y_m}$$ are in almost normal form, so clearly we have $$\overline{z_1}~\cdots~\overline{z_s}~\overline{x_1}~\cdots~\overline{x_n}~\mathcal{\widetilde{R}}~\overline{z_1}~\mathcal{R}~ \overline{z_1}~\cdots~\overline{z_s}~\overline{y_1}~\cdots~\overline{y_m}.$$

(ii) If $\beta_1=\alpha_1\leq \gamma_1$, then by Lemma \ref{product} $$\overline{z_1}~\cdots~\overline{z_s}~\overline{x_1}~\cdots~\overline{x_n}=
\overline{z_1}~\cdots~\overline{z_{j_{v}}}~\overline{z_{j_{v}+1}\cdots z_sx_1z_s\cdots z_{j_{v}+1}}~\cdots~\overline{z_sx_1z_s}~\overline{x_1}~\cdots~\overline{x_n}$$
and $$\overline{z_1}~\cdots~\overline{z_s}~\overline{y_1}~\cdots~\overline{y_m}=
\overline{z_1}~\cdots~\overline{z_{j_{v}}}~\overline{z_{j_{v}+1}\cdots z_sy_1z_s\cdots z_{j_{v}+1}}~\cdots~\overline{z_sy_1z_s}~\overline{y_1}~\cdots~\overline{y_m}$$
where $v\in [0,t-1]$, $\gamma_{v+1}, \cdots, \gamma_t\geq \alpha_1=\beta_1$ and $\gamma_v, \beta_1$ are incomparable or $v=0$. Note that the right hand sides are in almost normal form.

If $v\geq 1$, then clearly the required result is true, as the above two almost normal forms begin with the same idempotent. If $v=0$, then we need to show that $$z_1\cdots z_sx_1z_{s}\cdots z_{1}~\mathcal{R}~z_1\cdots z_sy_1z_{s}\cdots z_{1}.$$ 
Since $x_1~\mathcal{R}~y_1$, it follows from the structure of $B$ that $$z_1\cdots z_sx_1z_{s}\cdots z_{1}~\mathcal{R}~z_{1}\cdots z_sx_1~\mathcal{R}~z_1\cdots z_sy_1~\mathcal{R}~z_1\cdots z_sy_1z_{s}\cdots z_{1}$$ as required.

(iii) If $\beta=\alpha_1\geq \gamma_1$, then by Lemma \ref{product} $$\overline{z_1}~\cdots~\overline{z_s}~\overline{x_1}~\cdots~\overline{x_n}=
\overline{z_1}~\cdots~\overline{z_{s}}~\overline{x_1z_sx_1}~\cdots~\overline{x_{i_k}\cdots x_1z_sx_1\cdots x_{i_k}}~\overline{x_{i_k+1}}~\cdots~\overline{x_n}$$
and $$\overline{z_1}~\cdots~\overline{z_s}~\overline{y_1}~\cdots~\overline{y_m}=
\overline{z_1}~\cdots~\overline{z_{s}}~\overline{y_1z_sy_1}~\cdots~\overline{y_{l_p}\cdots y_1z_sy_1\cdots y_{l_p}}~\overline{y_{l_p+1}}~\cdots~\overline{y_m},$$
where $k\in [1,r],$ $\alpha_1, \cdots, \alpha_k\geq \gamma_1$, and $\alpha_{k+1}, \gamma_1$ are incomparable or $k=r,$ and $p\in [1,u],$ $\beta_1, \cdots, \beta_p\geq \gamma_1$, and $\beta_{p+1}, \gamma_1$ are incomparable or $p=u.$ Clearly, the right hand sides are in almost normal form, so that
$$\overline{z_1}~\cdots~\overline{z_s}~\overline{x_1}~\cdots~\overline{x_n}~\mathcal{\widetilde{R}}~\overline{z_1}~\mathcal{\widetilde{R}}~ \overline{z_1}~\cdots~\overline{z_s}~\overline{y_1}~\cdots~\overline{y_m}.$$

Similarly, we can show that $\widetilde{\mathcal{L}}$ is a right congruence, so that $\ig(B)$ is a weakly abundant semigroup satisfying the congruence condition. This completes the proof.\end{proof}

We finish this section by constructing a band $B$ for which $\ig(B)$ is not an abundant semigroup.

\begin{ex} {\rm Let $B=B_{\alpha}\cup B_{\beta}\cup B_\gamma$ be a band with semilattice decomposition structure and multiplication table defined by
\[\begin{array}{cc}
\begin{array}{c|cccc}
  & a & b & x & y\\ \hline
a & a & y & x & y \\
b & y & b & x & y\\
x & x & y & x & y \\
y & y & y & x & y
\end{array}&

\begin{tikzpicture}[scale=0.5]
\node (b) at (-3.5,0) {$B_\alpha$};
\node (x) at (-2,0) {\boxed{a}};
\node (y) at (2,0) {\boxed{b}};
\node (c) at (3.0,0) {$\, B_\beta$};
\node (z) at (0,-2) {\begin{tabular}{ r|c|c| }
 \cline{2-3}
 & $x$ & $y$ \\
 \cline{2-3}
 \end{tabular}};
\node (a) at (0,-3.0) {$B_\gamma$};
\path[-,font=\scriptsize,>=angle 60]
(x) edge node[right]{} (z)
(y) edge node[above]{} (z);
\end{tikzpicture}
\end{array}\]

First, it is easy to check that $B$ is indeed a semigroup. We now show that $\ig(B)$ is not abundant by arguing that the element  $\overline{a}~\overline{b}\in \ig(B)$ is not $\mathcal{R}^*$-related to any idempotent of $\ig(B)$. It follows from Theorem \ref{IG(band)} that $\overline{a}~\overline{b}~\mathcal{\widetilde{R}}~\overline{a}.$  However, $\overline{a}~\overline{b}$ is not $\ar^*$-related to $\overline{a}$, because $$\overline{x}~\overline{a}~\overline{b}=\overline{y}=\overline{y}~\overline{a}~\overline{b} \mbox{~but~} \overline{x}~\overline{a}=\overline{x}\neq \overline{y}=\overline{y}~\overline{a}.$$ From Lemma \ref{observation1}, $\overline{a}~\overline{b}$ is not $\mathcal{R}^*$-related to any idempotent of $\overline{B}$, and hence $\ig(B)$ is not an abundant semigroup.}
\end{ex}

\section{Free idempotent generated generated semigroups over locally large bands}\label{sec:locally large}

We recall from the Introduction that if $B=\bigcup_{\alpha\in Y} B_{\alpha}$ is a semilattice $Y$ of rectangular bands $B_\alpha, \alpha\in Y$, then $B$ is  {\em locally large}  if for all $\alpha, \beta\in Y$ with $\beta>\alpha$, $u\in B_\alpha$ and $v\in B_\beta$, we have $uv=vu=u$. Clearly $B$ is locally large if and only if for any $e\in B$, the local subsemigroup
$eBe$ is as large as is possible in the sense that for $e\in B_{\alpha}$, we have $eBe=\{ e\}\cup \bigcup_{\beta<\alpha}B_{\beta}$.  
In this section we show that  the word problem of $\ig(B)$ is solvable for a locally large band $B$. Further, in Section \ref{sec:(P)}, we will show that for any such $B$, the semigroup $\ig(B)$ is abundant.

It is easy to see that if $B$ is locally large, then for any $\alpha,\beta\in Y$  with $\alpha<\beta$, $u\in B_\alpha$ and $v\in B_\beta$, the products $uv$ and $vu$ are basic. We also note that any locally large band $B$ lies in the variety of
{\em regular bands}, that is, it satisfies the identity $xyxzx=xyzx$. To see this, let $x\in B_{\alpha}, y\in B_{\beta}$ and $z\in B_{\gamma}$. If $\alpha=\alpha\beta\gamma$, clearly $xyxzx=x=xyzx$. Otherwise, $\alpha>\alpha\beta\gamma$ and
\[xyxzx=x(yxz)(yxz)x=xyx(zyx)zx=xy(zyx)zx=(xyzy)xzx=xyzyzx=xyzx.\]
On the other hand, it is easy to see that if  $B$ is locally large, then  $B$ is normal if and only if $B_{\alpha}$ is trivial for all non-maximal $\alpha\in Y$. 

\begin{lem}\label{locally large-1}
Let $B$ be a locally large band, and let $\overline{x_1}~\cdots~\overline{x_n}$, $\overline{y_1}~\cdots~\overline{y_m}\in \ig(B)$ have left to right significant indices $i_1,\cdots,i_r$ and $j_1,\cdots,j_r,$ respectively. If $\overline{x_1}~\cdots~\overline{x_n}=\overline{y_1}~\cdots~\overline{y_m},$ then for any $l\in [1,r]$, $\overline{x_1}~\cdots~\overline{x_{i_l}}=\overline{y_1}~\cdots~\overline{y_{j_l}}.$
\end{lem}
\begin{proof}
Suppose that $x_i\in B_{\alpha_i}$ for all $i\in [1,r].$ It is enough to consider a single step, so suppose that $$\overline{x_1}~\cdots~\overline{x_n}\sim\overline{y_1}~\cdots~\overline{y_m}.$$ By Lemma \ref{useful}, for any $l\in [1,r]$, we have
$$\overline{y_1}~\cdots~\overline{y_{j_l}}=\overline{x_1}~\cdots~\overline{x_{i_l}}~\overline{u}$$ and $y_{j_l}=u'x_{i_l}u,$ where $u'=\varepsilon$ or $u'\in B_\sigma$ with $\sigma\geq \alpha_{i_l},$ and either $u=\varepsilon$,  or $u\in B_\delta$ for some $\delta> \alpha_{i_l}$, or $u\in B_{\alpha_{i_l}}$ and there exists $v\in B_{\theta}$ with $\theta>\alpha_{i_l}$, $vu=u$ and $uv=x_{i_l}$. By the comment proceeding Lemma \ref{locally large-1} we see that in each case, $\overline{x_{i_l}}~\overline{u}=\overline{x_{i_l}},$ so that clearly, $\overline{y_1}~\cdots~\overline{y_{j_l}}=\overline{x_1}~\cdots~\overline{x_{i_l}}.$
\end{proof}

\begin{lem}\label{segment-equal}
Let $B$ be a locally large band, let $\overline{x_1}~\cdots~\overline{x_n}\in \ig(B)$ be in almost normal form with $Y$-length $r$, left to right significant indices $i_1, \cdots, i_r=n$ and ordered $Y$-components $\alpha_1, \cdots, \alpha_r,$ and let $\overline{y_1}~\cdots~\overline{y_m}\in \ig(B)$ be in almost normal form with $Y$-length $s$, left to right significant indices $j_1, \cdots, j_s=m$ and ordered $Y$-components $\beta_1, \cdots, \beta_s.$ Then \[\overline{x_1}~\cdots~\overline{x_n}=\overline{y_1}~\cdots~\overline{y_m}\] in $\ig(B)$ if and only if $r=s,$ $\alpha_l=\beta_l$ and $\overline{x_{i_{l-1}+1}}~\cdots~\overline{x_{i_{l}}}=\overline{y_{j_{l-1}+1}}~\cdots~\overline{x_{j_{l}}}$ in $\ig(B),$ for each $l\in [1,r],$  where $i_0=j_0=0$.
 \end{lem}

\begin{proof}
The sufficiency is obvious. Suppose now that $\overline{x_1}~\cdots~\overline{x_n}=\overline{y_1}~\cdots~\overline{y_m}$ in $\ig(B)$. Then  it follows from Lemma \ref{indices are equal} that $r=s$ and $\alpha_i=\beta_i$ for all $i\in [1,r].$  From Lemma \ref{locally large-1}, we have that $\overline{x_1}~\cdots~\overline{x_{i_l}}=\overline{y_1}~\cdots~\overline{y_{j_l}}$ in $\ig(B),$ for all $l\in [1,r].$ Then by the dual of Lemma \ref{locally large-1}, $\overline{x_{i_{l-1}+1}}~\cdots~\overline{x_{i_{l}}}=\overline{y_{j_{l-1}+1}}~\cdots~\overline{x_{j_{l}}}$ in $\ig(B).$
 \end{proof}

\begin{lem}\label{basic lem}
Let $B$ be a locally large band  and $w=\overline{x_1}~\cdots~\overline{x_n}\in \overline{B}^+$ with $x_i\in B_{\alpha_i}$ for each $i\in [1,n].$ Suppose that there exists an $\alpha\in Y$ such that for all $i\in [1,n]$, $\alpha_i\geq \alpha$ and there is at least one $j\in [1,n]$ such that $\alpha=\alpha_j.$ Suppose also that $p\in\overline{B}^+$ and $w\,\sim\, p$.  Then  $w'=p'$ in $\ig(B_\alpha),$ where $w'$ and $p'$ are words obtained by deleting all letters in $w$ and $p$ which do not lie in $B_\alpha.$
 \end{lem}

\begin{proof}
Suppose that we split $x_k=uv$ for some $k\in [1,n],$ where $u\in B_{\nu}$ and $v\in B_{\tau}.$ Then we have $$w=\overline{x_1}~\cdots~\overline{x_{k-1}}~\overline{x_k}~\overline{x_{k+1}}~\cdots~\overline{x_n}\sim \overline{x_1}~\cdots~\overline{x_{k-1}}~\overline{u}~\overline{v}~\overline{x_{k+1}}~\cdots~\overline{x_n}=p.$$
 If $\alpha_k>\alpha,$ then $\nu, \tau>\alpha.$ Hence $w'=p'$ in $\overline{B_\alpha}^+;$ of course, they are also equal in $\ig(B_\alpha).$

If $\alpha_k=\alpha$ and $\mu=\tau=\alpha,$ then $u~\mathcal{L}~v$ or $u~\mathcal{R}~v,$ so that $uv$ is basic in $B_\alpha$. In this case, $\overline{x_k}=\overline{uv}=\overline{u}~\overline{v}$ in $\ig(B_\alpha),$ so that certainly, $$p'=(\overline{x_1}~\cdots~\overline{x_{k-1}})'~\overline{u}~\overline{v}~(\overline{x_{k+1}}~\cdots~\overline{x_n})'=
(\overline{x_1}~\cdots~\overline{x_{k-1}})'~\overline{x_k}~(\overline{x_{k+1}}~\cdots~\overline{x_n})'=w'$$ in $\ig(B_\alpha).$

If $\alpha_k=\alpha$ and $\nu>\tau=\alpha,$ then we have $x_k=uv=v$ as $B$ is a locally large band, so that $$\begin{aligned}
p'&=(\overline{x_1}~\cdots~\overline{x_{k-1}})'~(\overline{u}~\overline{v})'~(\overline{x_{k+1}}~\cdots~\overline{x_n})'\\
&=
(\overline{x_1}~\cdots~\overline{x_{k-1}})'~\overline{v}~(\overline{x_{k+1}}~\cdots~\overline{x_n})'\\
&=
(\overline{x_1}~\cdots~\overline{x_{k-1}})'~\overline{x_k}~(\overline{x_{k+1}}~\cdots~\overline{x_n})'\\
&=
w'
\end{aligned}$$ in $\overline{B_\alpha}^+$ (where either or both of the left and right factors can be empty), so that certainly $p'=w'$ in $\ig(B_\alpha).$

A similar argument holds if $\alpha_k=\alpha$ and $\alpha=\nu<\tau.$
\end{proof}
\begin{lem}\label{word problem lem}
Let $B$ be a locally large band and let $x_1, \cdots, x_n, y_1, \cdots, y_m\in B_\alpha$ for some $\alpha\in Y.$ Then with $w=\overline{x_1}~\cdots~\overline{x_n}$ and $p=\overline{y_1}~\cdots~\overline{y_m}$ we have $w=p$ in $\ig(B_\alpha)$ if and only if $w=p$ in $\ig(B).$
 \end{lem}
\begin{proof}
The sufficiency is clear, as any basic pair in $B_\alpha$ is basic in $B.$ Conversely, if $w=p$ in $\ig(B)$, there exists a finite sequence $$w=w_0\sim w_1\sim w_2\cdots\sim w_{s-1}\sim w_s=p.$$ Let $w_0',w_1', \cdots,  w_s'$ be the words obtained from $w_0, w_1,  \cdots, w_s$ by deleting letters $x$ within the word such that $x\in B_\beta$ with $\beta\neq\alpha.$ From Lemma \ref{basic lem}, we have that $w_0'=w_1'=w_2'=\cdots=w_{s-1}'=w_s'$ in $\ig(B_\alpha).$ Note that $w_0'=w_0=w\in \overline{B_\alpha}^+$ and $w_s'=w_s=p\in \overline{B_\alpha}^+,$ so that $w=p$ in $\ig(B_\alpha).$
 \end{proof}

\begin{them}
Let $B$ be a locally large band. Then the word problem of $\ig(B)$ is solvable.
\end{them}
\begin{proof}
The result is immediate from Lemmas \ref{rect-unique}, \ref{segment-equal} and \ref{word problem lem}.
 \end{proof}

\section{Free idempotent generated semigroups with condition (P)}\label{sec:(P)}

We have shown that for any band $B$, the semigroup $\ig(B)$ is always weakly abundant with the congruence condition, but not necessarily abundant. This section is devoted to finding some special kinds of bands $B$ for which $\ig(B)$ is abundant. As a means to this end we introduce a technical condition. 

\begin{definition}
We say that the semigroup $\ig(B)$ satisfies {\it Condition} $(P)$ if for any two almost normal forms $\overline{u_{1}}~\cdots~\overline{u_n}=
\overline{v_{1}}~\cdots~\overline{v_m}\in \ig(B)$ with $Y$-length $r$, left to right significant indices $i_1, \cdots, i_r=n$ and $l_1, \cdots, l_r=m,$ respectively,  the following statements (with $i_0=l_0=0$) hold:

(i) $u_{i_s}~\mathcal{L}~v_{l_s}$ implies $\overline{u_1}~\cdots \overline{u_{i_s}}=\overline{v_1}~\cdots~\overline{v_{l_s}}$, for all $s\in [1,r].$

(ii) $u_{i_t+1}~\mathcal{R}~v_{l_t+1}$ implies $\overline{u_{i_t+1}}~\cdots \overline{u_{n}}=\overline{v_{l_t+1}}~\cdots~\overline{v_{m}}$, for all $t\in [0, r-1].$
\end{definition}

It follows from Lemma~\ref{locally large-1} that Condition (P) holds for $\ig(B)$, where $B$ is a locally large band. On the other hand, it is a consequence of our results and  Example~\ref{ex:final} that not every band has Condition (P), in particular, not every normal band has Condition (P).

\begin{pro}\label{simple normal band}
Let $B$ be a band for which $\ig(B)$ satisfies Condition $(P)$. In addition, suppose that $B$ is normal $($so that $B=\mathcal{B}(Y; B_\alpha, \phi_{\alpha,\beta})$$)$ or locally large. Then $\ig(B)$ is an abundant semigroup.
\end{pro}
\begin{proof}
Let $\overline{x_1}~\cdots~\overline{x_{n}}\in \ig(B)$ be in almost normal form with $Y$-length $r$, left to right significant indices $i_1, \cdots, i_r=n$, and ordered $Y$-components $\alpha_1, \cdots, \alpha_r.$ By Theorem \ref{IG(band)},  $\overline{x_1}~\cdots~\overline{x_{i_r}}~\mathcal{\widetilde{R}}~\overline{x_1}$. We aim to show that $\overline{x_1}~\cdots~\overline{x_{i_r}}~\mathcal{R}^*~\overline{x_1}.$ From Lemma \ref {observation2}, we only need to show that for any two almost normal forms $\overline{y_1}~\cdots~\overline{y_{m}}, \overline{z_1}~\cdots~\overline{z_h}\in \ig(B)$
we have
\[\overline{z_1}~\cdots~\overline{z_{h}}
~\overline{x_1}~\cdots~\overline{x_{n}}=\overline{y_1}~\cdots~\overline{y_{m}}
~\overline{x_1}~\cdots~\overline{x_{n}}\Longrightarrow \overline{z_1}~\cdots~\overline{z_{h}}
~\overline{x_1}=\overline{y_1}~\cdots~\overline{y_{m}}
~\overline{x_1}.\]

Suppose that $\overline{y_1}~\cdots~\overline{y_{m}}$ has $Y$-length $m$, left to right significant indices $l_1, \cdots, l_s=m$, and ordered $Y$-components $\beta_1, \cdots, \beta_s,$ and $\overline{z_1}~\cdots~\overline{z_h}\in \ig(B)$ has $Y$-length $t$, left to right significant indices $j_1, \cdots, j_t=h$, and ordered $Y$-components $\gamma_1, \cdots, \gamma_t$.

Assume now that $$\overline{z_1}~\cdots~\overline{z_{j_t}}
~\overline{x_1}~\cdots~\overline{x_{i_r}}=\overline{y_1}~\cdots~\overline{y_{l_s}}
~\overline{x_1}~\cdots~\overline{x_{i_r}}$$ 
(it will be convenient to use the indices $i_r,l_s,j_t$). We consider the following cases:

\medskip

(i) If $\gamma_t, \alpha_1$ and $\beta_s, \alpha_1$ are incomparable, then both sides of the above equality are in almost normal form, so that by Condition $(P)$ $$\overline{z_1}~\cdots~\overline{z_{j_t}}
~\overline{x_1}~\cdots~\overline{x_{i_1}}=\overline{y_1}~\cdots~\overline{y_{l_s}}
~\overline{x_1}~\cdots~\overline{x_{i_1}}.$$
Since $\overline{x_{1}}\cdots~\overline{ x_{i_1}}~\mathcal{R}~\overline{x_{i_1}}$ by Lemma \ref{rectangular band}, we have  $\overline{z_1}~\cdots~\overline{z_{j_t}}
~\overline{x_1}=\overline{y_1}~\cdots~\overline{y_{l_s}}
~\overline{x_1}.$

\medskip

(ii) If $\gamma_t\leq \alpha_1$ and $\beta_s, \alpha_1$ are incomparable, then by Lemma \ref{product}, $\overline{z_1}~\cdots~\overline{z_{j_t}}
~\overline{x_1}~\cdots~\overline{x_{i_r}}$ has an almost normal form  $$\overline{z_1}~\cdots~\overline{z_{j_t}}~\overline{x_1z_{j_t}x_1}~\cdots~
\overline{x_{i_v}\cdots x_{1}z_{j_t}x_1\cdots x_{i_v}}~\overline{x_{i_v+1}}~\cdots~\overline{x_{i_r}}, $$ for some $v\in [1,r]$, where $\gamma_t\leq \alpha_1, \cdots, \alpha_v$ and $v=r$ or $\gamma_t, \alpha_{v+1}$ are incomparable. Hence we have $$\overline{z_1}~\cdots~\overline{z_{j_t}}~\overline{x_1z_{j_t}x_1}~\cdots~
\overline{x_{i_v}\cdots x_{1}z_{j_t}x_1\cdots x_{i_v}}~\overline{x_{i_v+1}}~\cdots~\overline{x_{i_r}}=\overline{y_1}~\cdots~\overline{y_{l_s}}
~\overline{x_1}~\cdots~\overline{x_{i_r}}.$$ Note that both sides of the above equality are in almost normal form. It follows from Corollary \ref{direct} that $$(\overline{z_1}~\cdots~\overline{z_{j_t}}~\overline{x_1z_{j_t}x_1}~\cdots~
\overline{x_{i_v}\cdots x_{1}z_{j_t}x_1\cdots x_{i_v}}~\overline{x_{i_v+1}}~\cdots~\overline{x_{i_r}})~\overline{\theta}=(\overline{y_1}~\cdots~\overline{y_{l_s}}
~\overline{x_1}~\cdots~\overline{x_{i_r}})~\overline{\theta}$$ and so $$\overline{\gamma_1}~\cdots~\overline{\gamma_t}~\overline{\alpha_{v+1}}~\cdots~\overline{\alpha_r}
=\overline{\beta_1}~\cdots~\overline{\beta_s}~\overline{\alpha_1}~\cdots~\overline{\alpha_r}.$$ Since $v\geq 1$, we have $\gamma_t=\alpha_v$. To avoid contradiction, $v=1$,  and hence by Condition $(P)$ $$ \overline{z_1}~\cdots~\overline{z_{j_t}}~\overline{x_1z_{j_t}x_1}~\cdots~
\overline{x_{i_1}\cdots x_{1}z_{j_t}x_1\cdots x_{i_1}}=\overline{y_1}~\cdots~\overline{y_{l_s}}
~\overline{x_1}~\cdots~\overline{x_{i_1}}.$$ As $\gamma_t=\alpha_v$, $$\overline{z_1}~\cdots~\overline{z_{j_t}}~\overline{x_1}~\cdots~\overline{x_{i_1}}=\overline{y_1}~\cdots~\overline{y_{l_s}}
~\overline{x_1}~\cdots~\overline{x_{i_1}}$$ so that $\overline{z_1}~\cdots~\overline{z_{j_t}}~\overline{x_1}=\overline{y_1}~\cdots~\overline{y_{l_1}}~\cdots~\overline{y_{l_s}}
~\overline{x_1}.$

\medskip

(iii) If $\gamma_t\leq \alpha_1$ and $\beta_s\leq \alpha_1$, then by Lemma \ref{product} we have the following two almost normal forms for $\overline{z_1}~\cdots~\overline{z_{j_t}}
~\overline{x_1}~\cdots~\overline{x_{i_r}}$ and  $\overline{y_1}~\cdots~\overline{y_{l_s}}
~\overline{x_1}~\cdots~\overline{x_{i_r}},$ namely,
$$\overline{z_1}~\cdots~\overline{z_{j_t}}~\overline{x_1z_{j_t}x_1}~\cdots~\overline{x_{i_v}\cdots x_1z_{j_t}x_1\cdots x_{i_v}}~\overline{x_{i_v+1}}~\cdots~\overline{x_{i_r}}$$
where $v\in [1,r]$ such that $\gamma_t\leq \alpha_1, \cdots, \alpha_v$ and $v=r$ or $\gamma_t, \alpha_{v+1}$ are incomparable, and
$$\overline{y_1}~\cdots~\overline{y_{l_s}}~\overline{x_1y_{l_s}x_1}~\cdots~\overline{x_{i_u}\cdots x_1y_{l_s}x_1\cdots x_{i_u}}~\overline{x_{i_u+1}}~\cdots~\overline{x_{i_r}}$$
where  $u\in [1,r]$ with $\beta_s\leq \alpha_1, \cdots, \alpha_u$ and $u=r$ or $\beta_s, \alpha_{u+1}$ are incomparable.
Hence by Corollary \ref{direct}, $$\overline{\gamma_1}~\cdots~\overline{\gamma_t}~\overline{\alpha_{v+1}}~\cdots~\overline{\alpha_r}
=\overline{\beta_1}~\cdots~\overline{\beta_s}~\overline{\alpha_{u+1}}~\cdots~\overline{\alpha_r}$$
If $v>u$, then $\gamma_t=\alpha_v$, to avoid contradiction $v=1$, so $u=0$, contradiction. Similarly, $v<u$ is impossible. We deduce that $v=u$, and so $t=s$ and $\beta_s=\gamma_t$.  

If $B$ is a normal band satisfying Condition $(P),$ $$x_1z_{j_t}x_1=x_1\phi_{\alpha_1,\gamma_t}=x_1\phi_{\alpha_1,\beta_s}=x_1y_{l_s}x_1$$ $$\vdots$$ $$x_{i_v}\cdots x_1z_{j_t}x_1\cdots x_{i_v}=x_{i_v}\phi_{\alpha_v,\gamma_t}=x_{i_u}\phi_{\alpha_u,\beta_s}=x_{i_u}\cdots x_{1}y_{l_s}x_{1}\cdots x_{i_u}$$ so that by Condition $(P)$, we have $$\overline{z_1}~\cdots~\overline{z_{j_t}}~\overline{x_1z_{j_t}x_1}~\cdots~\overline{x_{i_v}\cdots x_1z_{j_t}x_1\cdots x_{i_v}}=\overline{y_1}~\cdots~\overline{y_{l_s}}~\overline{x_1y_{l_s}x_1}~\cdots~\overline{x_{i_u}\cdots x_1y_{l_s}x_1\cdots x_{i_u}}.$$
On the other hand, we have $$\overline{x_1z_{j_t}x_1}~\cdots~\overline{x_{i_v}\cdots x_1z_{j_t}x_1\cdots x_{i_v}}=\overline{x_1y_{l_s}x_1}~\cdots~\overline{x_{i_u}\cdots x_1y_{l_s}x_1\cdots x_{i_u}}$$ which by Lemma~\ref{rectangular band} is $\mathcal{R}$-related to $\overline{x_1z_{j_t}x_1}$ in $\ig(B_{\gamma_t})$ and hence in $\ig(B)$, so that $$\overline{z_1}~\cdots~\overline{z_{j_t}}~\overline{x_1z_{j_t}x_1}=\overline{y_1}~\cdots~\overline{y_{l_s}}~\overline{x_1y_{l_s}x_1},$$ and hence
$$\overline{z_1}~\cdots~\overline{z_{j_t}}~\overline{x_1}=\overline{y_1}~\cdots~\overline{y_{l_s}}~\overline{x_1}.$$

Suppose now that $B$ is a locally large band. Consider first the case where $v=u=1.$ By Lemma \ref{locally large-1} we have $$\overline{z_1}~\cdots~\overline{z_{j_t}}~\overline{x_1z_{j_t}x_1}~\cdots~
\overline{x_{i_1}\cdots x_1z_{j_t}x_1\cdots x_{i_1}}=\overline{y_1}~\cdots~\overline{y_{l_s}}~\overline{x_1y_{l_s}x_1}~\cdots~\overline{x_{i_1}\cdots x_1y_{l_s}x_1\cdots x_{i_1}}$$ and so $$\overline{z_1}~\cdots~\overline{z_{j_t}}~\overline{x_1}~\cdots~
\overline{x_{i_1}}=\overline{y_1}~\cdots~\overline{y_{l_s}}~\overline{x_1}~\cdots~\overline{x_{i_1}}$$
so that $$\overline{z_1}~\cdots~\overline{z_{j_t}}~\overline{x_1}=\overline{y_1}~\cdots~\overline{y_{l_s}}~\overline{x_1}.$$
Suppose now that $v=u>1.$ By assumption $\beta_s=\gamma_t\leq \alpha_1, \cdots, \alpha_v$. We claim that there exists no $j\in [1,v]$ such that $\gamma_t=\alpha_j;$ otherwise we will have $\alpha_j, \alpha_{j+1}$ are comparable if $v>j$ or $\alpha_v, \alpha_{v-1}$ are comparable if $v=j.$ Hence $\gamma_t=\beta_s< \alpha_1, \cdots, \alpha_v.$  Since $B$ is a locally large band, we have $$\overline{z_1}~\cdots~\overline{z_{j_t}}~\overline{x_1z_{j_t}x_1}~\cdots~
\overline{x_{i_v}\cdots x_1z_{j_t}x_1\cdots x_{i_v}}~\overline{x_{i_v+1}}~\cdots~\overline{x_{i_r}}=\overline{z_1}~\cdots~\overline{z_{j_t}}~\overline{x_{i_v+1}}~\cdots~\overline{x_{i_r}}$$
and $$\overline{y_1}~\cdots~\overline{y_{l_s}}~\overline{x_1y_{l_s}x_1}~\cdots~\overline{x_{i_u}\cdots x_1y_{l_s}x_1\cdots x_{i_u}}~\overline{x_{i_u+1}}~\cdots~\overline{x_{i_r}}=\overline{y_1}~\cdots~\overline{y_{l_s}}~\overline{x_{i_v+1}}~\cdots~\overline{x_{i_r}}$$
so that it follows from Lemma \ref{locally large-1} that $$\overline{z_1}~\cdots~\overline{z_{j_t}}=\overline{y_1}~
\cdots~\overline{y_{l_s}}$$ and so certainly $$\overline{z_1}~\cdots~\overline{z_{j_t}}~\overline{x_1}=\overline{y_1}~
\cdots~\overline{y_{l_s}}~\overline{x_1}.$$

\medskip

(iv) If $\gamma_t\leq \alpha_1$ and $\beta_s\geq \alpha_1$, then by Lemma \ref{product} we have the following two almost normal forms for $\overline{z_1}~\cdots~\overline{z_{j_t}}
~\overline{x_1}~\cdots~\overline{x_{i_r}}$ and  $\overline{y_1}~\cdots~\overline{y_{l_s}}
~\overline{x_1}~\cdots~\overline{x_{i_r}},$ namely,
$$\overline{z_1}~\cdots~\overline{z_{j_t}}~\overline{x_1z_{j_t}x_1}~\cdots~\overline{x_{i_v}\cdots x_1z_{j_t}x_1\cdots x_{i_v}}~\overline{x_{i_v+1}}~\cdots~\overline{x_{i_r}}$$
for some $v\in [1,r]$ with $\gamma_t\leq \alpha_1, \cdots, \alpha_v$ and $v=r$ or $\gamma_t, \alpha_{v+1}$ are incomparable, and
$$\overline{y_1}\cdots \overline{y_{l_u}}~\overline{y_{l_u+1}\cdots y_{l_s}x_1y_{l_s}\cdots y_{l_u+1}}~\cdots~\overline{y_{l_s}x_1y_{l_s}}~\overline{x_1}~\cdots~\overline{x_{i_r}}$$ for some $u\in [0,s-1]$ with $\beta_{u+1}, \cdots, \beta_s\geq \alpha_1$ and $\beta_u, \alpha_1$ are incomparable or $u=0.$ It follows from Corollary \ref{direct} that
$$\overline{\gamma_1}~\cdots~\overline{\gamma_t}~\overline{\alpha_{v+1}}~\cdots~\overline{\alpha_r}
=\overline{\beta_1}~\cdots~\overline{\beta_u}~\overline{\alpha_1}~\cdots~\overline{\alpha_r}.$$ Note that both sides of the above equality are normal forms of $\ig(Y).$
As $v\geq 1$, we have $\gamma_t=\alpha_v$, so that to avoid contradiction we have $v=1$ and  then $x_{i_1}\cdots x_1z_{j_t}x_1\cdots x_{i_1}=x_{i_1}$. Hence by Condition (P)
$$\overline{z_1}~\cdots~\overline{z_{j_t}}~\overline{x_1z_{j_t}x_1}~\cdots~\overline{x_{i_1}\cdots x_1z_{j_t}x_1\cdots x_{i_1}}$$$$=\overline{y_1}\cdots \overline{y_{l_u}}~\overline{y_{l_u+1}\cdots y_{l_s}x_1y_{l_s}\cdots y_{l_u+1}}~\cdots~
\overline{y_{l_s}x_1y_{l_s}}~\overline{x_1}~\cdots~\overline{x_{i_1}}$$ and so $$\overline{z_1}~\cdots~\overline{z_{j_t}}~\overline{x_1}~\cdots~\overline{x_{i_1}}=\overline{y_1}\cdots \overline{y_{l_s}}~\overline{x_1}~\cdots~\overline{x_{i_1}},$$ which implies $\overline{z_1}~\cdots~\overline{z_{j_t}}~\overline{x_1}=\overline{y_1}\cdots \overline{y_{l_s}}~\overline{x_1}$.

\medskip

(v) If $\gamma_t\geq \alpha_1$ and $\beta_s\geq \alpha_1$, then by Lemma \ref{product} we have the following two almost normal forms for $\overline{z_1}~\cdots~\overline{z_{j_t}}
~\overline{x_1}~\cdots~\overline{x_{i_r}}$ and  $\overline{y_1}~\cdots~\overline{y_{l_s}}
~\overline{x_1}~\cdots~\overline{x_{i_r}},$ namely, $$\overline{z_1}\cdots \overline{z_{j_v}}~\overline{z_{j_v+1}\cdots z_{j_t}x_1z_{j_t}\cdots z_{j_v+1}}~\cdots~\overline{z_{j_t}x_1z_{j_t}}~\overline{x_1}~\cdots~\overline{x_{i_1}}~\cdots~\overline{x_{i_r}}$$ for some $v\in [0,t-1]$ such that $\gamma_{v+1}, \cdots, \gamma_t\geq \alpha_1$ and $\gamma_v, \alpha_1$ are incomparable or $v=0,$ and
$$\overline{y_1}\cdots \overline{y_{l_u}}~\overline{y_{l_u+1}\cdots y_{l_s}x_1y_{l_s}\cdots y_{l_u+1}}~\cdots~\overline{y_{l_s}x_1y_{l_s}}~\overline{x_1}~\cdots~\overline{x_{i_1}}~\cdots~\overline{x_{i_r}}$$ for some $u\in [0,s-1]$ such that $\beta_{u+1}, \cdots, \beta_s\geq \alpha_1$ and $\beta_u, \alpha_1$ are incomparable or $u=0$.
Hence by Condition $(P),$ $$\overline{z_1}\cdots \overline{z_{j_v}}~\overline{z_{j_v+1}\cdots z_{j_t}x_1z_{j_t}\cdots z_{j_v+1}}~\cdots~\overline{z_{j_t}x_1z_{j_t}}~\overline{x_1}~\cdots~\overline{x_{i_1}}$$$$=\overline{y_1}\cdots \overline{y_{l_u}}~\overline{y_{l_u+1}\cdots y_{l_s}x_1y_{l_s}\cdots y_{l_u+1}}~\cdots~\overline{y_{l_s}x_1y_{l_s}}~\overline{x_1}~\cdots~\overline{x_{i_1}},$$
so that $$\overline{z_1}\cdots~\overline{z_{j_t}}~\overline{x_1}~\cdots~\overline{x_{i_1}}=
\overline{y_1}\cdots~\overline{y_{l_s}}~\overline{x_1}~\cdots~\overline{x_{i_1}}$$
and hence $\overline{z_1}\cdots~\overline{z_{j_t}}~\overline{x_1}=\overline{y_1}\cdots~\overline{y_{l_s}}~\overline{x_1}.$

\medskip

(vi) If $\gamma_t\geq \alpha_1$ and $\beta_s, \alpha_1$ are incomparable, then by Lemma \ref{product} $$\overline{z_1}\cdots \overline{z_{j_v}}~\overline{z_{j_v+1}\cdots z_{j_t}x_1z_{j_t}\cdots z_{j_v+1}}~\cdots~\overline{z_{j_t}x_1z_{j_t}}~\overline{x_1}~\cdots~\overline{x_{i_1}}~\cdots~\overline{{x_{i_r}}}=\overline{y_1}\cdots \overline{y_{l_s}}~\overline{x_1}~\cdots~\overline{x_{i_1}}~\cdots~\overline{{x_{i_r}}}$$ for some $v\in [0,t-1]$ with $\gamma_{v+1}, \cdots, \gamma_t\geq \alpha_1$ and $\gamma_v, \alpha_1$ are incomparable or $v=0$. Note that both sides of the above equality are in almost normal form. Again by Condition (P) $$\overline{z_1}\cdots \overline{z_{j_v}}~\overline{z_{j_v+1}\cdots z_{j_t}x_1z_{j_t}\cdots z_{j_v+1}}~\cdots~\overline{z_{j_t}x_1z_{j_t}}~\overline{x_1}~\cdots~\overline{x_{i_1}}=\overline{y_1}\cdots \overline{y_{l_s}}~\overline{x_1}~\cdots~\overline{x_{i_1}}$$
so that
$$\overline{z_1}\cdots~\overline{z_{j_t}}~\overline{x_1}~\cdots~\overline{x_{i_1}}=
\overline{y_1}\cdots~\overline{y_{l_s}}~\overline{x_1}~\cdots~\overline{x_{i_1}}
$$
and hence $\overline{z_1}\cdots~\overline{z_{j_t}}~\overline{x_1}=\overline{y_1}\cdots~\overline{y_{l_s}}~\overline{x_1}.$

\medskip

From the above case-by-case analysis, we  deduce that $\overline{x_1}~\cdots~\overline{x_{i_r}}~\mathcal{R}^*~\overline{x_1}$, and similarly we can show that $\overline{x_1}~\cdots~\overline{x_{i_r}}~\mathcal{L}^*~\overline{x_{i_r}},$ so that $\ig(B)$ is an abundant semigroup.\end{proof}

We now aim to find examples of normal bands $B$ for which $\ig(B)$ satisfies Condition (P), so that by Proposition \ref{simple normal band}, $\ig(B)$ is abundant.

\medskip

A band $B=\bigcup_{\alpha\in Y} B_{\alpha}$ is called  {\it $Y$-basic} if it is a semilattice $Y$ of rectangular bands $B_\alpha$, $\alpha\in Y$, where $B_\alpha$ is either a left zero band or a right zero band. Any left or right regular band (that is, where  {\em every} $B_{\alpha}$ is left zero, or {\em every} $B_{\alpha}$ is right zero) is $Y$-basic, but the class of $Y$-basic bands is easily seen to be larger. We now justify the terminology.

\begin{lem}\label{transfer}
Let $B=\bigcup_{\alpha\in Y} B_{\alpha}$ be a  band. Then $B$ is $Y$-basic if and only if it has the property that for any  $e\in B_{\alpha}$ and $f\in B_{\beta}$ the pair $(e,f)$ being  basic pair $B$ is equivalent to  $(\alpha, \beta)$ being basic  in $Y$. 
\end{lem}
\begin{proof} Suppose that $B$ has the given property on basic pairs. For any $\alpha\in Y$ fix $e\in B_{\alpha}$; since
$(e,f)$ must be basic in $B$ for any $f\in B_{\alpha}$, clearly $B_{\alpha}$ is a left or a right zero semigroup.

Conversely, suppose that $B$ is $Y$-basic.  Let $e\in B_{\alpha}$ and $f\in B_{\beta}$. If $(e,f)$ is basic, certainly  so is $(\alpha,\beta)$. For the converse, without loss of generality, suppose that $\alpha\leq \beta$. Then $ef,fe\in B_{\alpha}$. As $B$ is a $Y$-basic band, we have $B_\alpha$ is either a left zero band  or a right zero band. If $B_\alpha$ is a left zero band, then $e(ef)=e$, i.e. $ef=e$, so $(e,f)$ is a basic pair. If $B_\alpha$ is a right zero band, then $(fe)e=e$, i.e. $fe=e$, which again implies that $(e,f)$ is a basic pair.
\end{proof}

It follows from Lemma \ref{transfer} that for a $Y$-basic band $B$, every element of $\ig(B)$ has an almost normal form (which of course need  not be unique), say, $\overline{x_1}~\cdots~\overline{x_n}$  with 
$x_i\in B_{\alpha_i}$ and $\alpha_i$ and $\alpha_{i+1}$ incomparable, for all $i\in [1, n-1].$

\begin{lem}\label{lem:semilatticebasic}
Let $B$ be a $Y$-basic band. Then $\ig(B)$ satisfies Condition $(P).$
\end{lem}

\begin{proof}
Let $\overline{x_{1}}~\cdots~\overline{x_n}=
\overline{y_{1}}~\cdots~\overline{y_m}\in \ig(B)$ be in almost normal form with $Y$-length $r$, left to right significant indices $i_1, \cdots, i_r=n$, $j_1, \cdots, j_r=m$, respectively, and ordered $Y$-components $\alpha_1, \cdots, \alpha_r.$ It then follows from Corollary \ref{important} that for any $s\in [1,r],$ either
\[\overline{y_1}~\cdots \overline{y_{j_s}}=\overline{x_1}~\cdots~\overline{x_{i_s}}\]
and we are done, or
\[\overline{y_1}~\cdots \overline{y_{j_s}}=\overline{x_1}~\cdots~\overline{x_{i_s}}~\overline{e_1}~\cdots~\overline{e_m}\] where for all $k\in [1,m],$ $e_k\in B_{\delta_k}$ with $\delta_k\geq \alpha_{s}$.  In this case by Lemma \ref{transfer}, we have $$\overline{x_{i_s}}~\overline{e_1}~\cdots~\overline{e_m}=\overline{x_{i_s}e_1\cdots e_m},$$ so that if we assume $x_{i_s}~\mathcal{L}~y_{j_s},$ then
$$\overline{y_1}~\cdots \overline{y_{j_s}}=\overline{y_1}~\cdots \overline{y_{j_s}}~\overline{x_{i_s}}=\overline{x_1}~\cdots~\overline{x_{i_s}e_1\cdots e_m}~\overline{x_{i_s}}=\overline{x_1}~\cdots~\overline{x_{i_s}e_1\cdots e_mx_{i_s}}=\overline{x_1}~\cdots~\overline{x_{i_s}}.$$

Together with the dual, we have shown that $\ig(B)$ satisfies Condition (P).
\end{proof}

Let $B=\mathcal{B}(Y;B_\alpha,\phi_{\alpha,\beta})$ be a normal band. Clearly $B$ is  {\em locally small} in the sense that the local submonoids $eBe$ are as small as they can be, that is, for $e\in B_{\alpha}$, we have
$eBe=\{ e\} \cup\{ e\phi_{\alpha,\beta}:\alpha >\beta\}=\{ e\phi_{\alpha,\beta}:\alpha \geq\beta\}$.
We say that $B$ is a {\it pliant} if for every $\alpha\in Y$, there exists an $a_\alpha\in B_\alpha$ such that for all $\beta>\alpha$ and $u\in B_{\beta}$, we have $u\phi_{\beta, \alpha}=a_\alpha$. 

\begin{lem}\label{lem:pliant}
Let  $B=\mathcal{B}(Y;B_\alpha,\phi_{\alpha,\beta})$  be a pliant normal band. Then $\ig(B)$ satisfies Condition $(P).$
\end{lem}

\begin{proof}
First note that since $B$ is a pliant normal band, there exists  $a_\alpha\in B_\alpha$ be such that for any $\beta>\alpha$ and $u\in B_\beta,$ $u\phi_{\beta, \alpha}=a_\alpha.$

Let $\overline{x_{1}}~\cdots~\overline{x_n}=
\overline{y_{1}}~\cdots~\overline{y_m}\in \ig(B)$ be in almost normal form with $Y$-length $r$, left to right significant indices $i_1, \cdots, i_r=n$, $j_1, \cdots, j_r=m$, respectively, and ordered $Y$-components $\alpha_1, \cdots, \alpha_r.$ We may assume from Corollary \ref{important} that $$\overline{y_{1}}~\cdots~\overline{y_{j_l}}=\overline{x_1}~\cdots~\overline{x_{i_l}}~\overline{u_1}~\cdots~\overline{u_s}$$ such that for all $k\in [1,s]$ we have  $u_k\in B_{\delta_k}$ with $\delta_k>\alpha_{l},$ so that $u_k\phi_{\delta_k, \alpha_{l}}=a_{\alpha_{l}}$; or $u_k\in B_{\alpha_{l}}$ with $v_k u_k=u_k$ for some $v_k\in B_{\eta_k}$ such that $\eta_k>\alpha_{l},$ and in this case we have $a_{\alpha_{l}}u_k=u_k$, so that $a_{\alpha_{l}}~\mathcal{R}~u_k.$ Thus the idempotents $u_1\phi_{\delta_1, \alpha_{l}}, \cdots, u_s\phi_{\delta_s, \alpha_{l}}$ are all $\mathcal{R}$-related, and so $$\overline{x_{i_l}}~\overline{u_1}~\cdots~\overline{u_s}=\overline{x_{i_l}}~\overline{u_1\phi_{\delta_1, \alpha_{l}}}~\cdots~\overline{u_s\phi_{\delta_s, \alpha_{l}}}=\overline{x_{i_l}}~\overline{u_1\phi_{\delta_1, \alpha_{l}}\cdots u_s\phi_{\delta_s, \alpha_{l}}}.$$
 On the other hand, again using Corollary~ \ref{important} we have $y_{j_l}=wx_{i_l}u_1\cdots u_s,$ 
 for some $w\in B_{\alpha_l}$. Hence if we assume that $x_{i_l}~\mathcal{L}~y_{j_l}$, then $x_{i_l}=x_{i_l}u_1\cdots u_s,$ and so $x_{i_l}=x_{i_l}(u_1\phi_{\delta_1, \alpha_{l}})\cdots (u_s\phi_{\delta_s, \alpha_{l}}),$ so that $$\overline{x_{i_l}}~\overline{u_1\phi_{\delta_1, \alpha_{l}}\cdots u_s\phi_{\delta_s, \alpha_{l}}}=\overline{x_{i_l}(u_1\phi_{\delta_1, \alpha_{l}})\cdots (u_s\phi_{\delta_s, \alpha_{l}})}=\overline{x_{i_l}}.$$ Hence $\overline{y_1}~\cdots~\overline{y_{j_l}}=\overline{x_1}~\cdots~\overline{x_{i_l}}$ as required.
\end{proof}

As an immediate consequence of Proposition~\ref{simple normal band} and Lemmas~\ref{lem:semilatticebasic} and ~\ref{lem:pliant} we have the following result.

\begin{them}
Let $B$ be a  normal band that is $Y$-basic or pliant. Then $\ig(B)$ is abundant.
\end{them}

\section{A normal band $B$ for which $\ig(B)$ is not abundant}\label{sec:normal bands}

From Section \ref{sec:(P)}, we know that the free idempotent idempotent generated semigroup $\ig(B)$ over a normal band $B$ satisfying Condition (P) is an abundant semigroup. Therefore, one would like to ask whether $\ig(B)$ is abundant for any normal band $B$. In this section we answer the question in the negative by constructing a 10-element normal band $B$ such that $\ig(B)$ is not abundant.

Throughout this section, we will use $\mathcal{B}(Y; B_\alpha, \phi_{\alpha,\beta})$ as standard notation for a normal band.

\begin{lem}\label{basic transfer-normal}
Let $B$ be a normal band, and let $x\in B_{\beta}, y\in B_\gamma$ with $\beta,\gamma\geq \alpha$. Then $(x,y)$ is a basic pair implies $(x\phi_{\beta,\alpha}, y\phi_{\gamma,\alpha})$ is a basic pair and $$(x\phi_{\beta,\alpha})(y\phi_{\gamma,\alpha})=(xy)\phi_{\delta, \alpha},$$ where $\delta$ is minimum of $\beta$ and $\gamma.$
\end{lem}
\begin{proof}
Let $(x,y)$ be a basic pair with $x\in B_{\beta}, y\in B_\gamma$. Then $\beta, \gamma$ are comparable. If $\beta\geq \gamma$, then we either have $xy=y$ or $yx=y$. If $xy=y$, then $(x\phi_{\beta, \gamma})y=y$, so $$y\phi_{\gamma, \alpha}=((x\phi_{\beta, \gamma})y)\phi_{\gamma, \alpha}=(x\phi_{\beta,\alpha})(y\phi_{\gamma,\alpha}),$$ so $(x\phi_{\beta,\alpha}, y\phi_{\gamma,\alpha})$ is a basic pair. If $yx=y$, then $y(x\phi_{\beta,\gamma})=y$, so $$y\phi_{\gamma,\alpha}=(y(x\phi_{\beta,\gamma}))\phi_{\gamma,\alpha}= (y\phi_{\gamma,\alpha})(x\phi_{\beta,\alpha}),$$ so that $(x\phi_{\beta,\alpha}, y\phi_{\gamma,\alpha})$ is a basic pair.

A similar argument holds if $\gamma\geq \beta$. The final part of the lemma is clear.
\end{proof}

\begin{lem}\label{equal}
Let $B$ be a normal band and let $\overline{u_1}~\cdots~\overline{u_n}\in \ig(B)$ with $u_i\in B_{\alpha_i}$ and $\alpha_i\geq \alpha$ for all $i\in [1,n]$. Suppose that $\overline{v_1}~\cdots~\overline{v_m}\in \ig(B)$ with $v_i\in B_{\beta_i}$ for all $i\in [1,m]$ and $\overline{u_1}~\cdots~\overline{u_n}\sim \overline{v_1}~\cdots~\overline{v_m}$. Note that $\beta_i\geq \alpha$, for all $i\in [1,m]$. Then in $\ig(B_\alpha)$ we have $$\overline{u_1\phi_{\alpha_1,\alpha}}~\cdots~\overline{u_n\phi_{\alpha_n,\alpha}}=\overline{v_1\phi_{\beta_1,\alpha}}~\cdots~
\overline{v_m\phi_{\beta_m,\alpha}}.$$
\end{lem}
\begin{proof}
Suppose that $u_i=xy$ is a basic product with $x\in B_{\delta}, y\in B_{\eta}$, for some $i\in [1,n]$. Note that the minimum of $\delta$ and $\eta$ is $\alpha_i$. Then $$\overline{u_1}~\cdots~\overline{u_n}\sim \overline{u_1}~\cdots~\overline{u_{i-1}}~\overline{x}~\overline{y}~\overline{u_{i+1}}~\cdots~\overline{u_n}.$$
If follows from Lemma \ref{basic transfer-normal} that in $\ig(B_\alpha)$
$$
\begin{aligned}
\overline{u_1\phi_{\alpha_1,\alpha}}~\cdots~\overline{u_n\phi_{\alpha_n,\alpha}}&=\overline{u_1\phi_{\alpha_1, \alpha}}~\cdots~\overline{u_{i-1}\phi_{\alpha_{i-1},\alpha}}~\overline{u_i\phi_{\alpha_i, \alpha}}~\overline{u_{i+1}\phi_{\alpha_{i+1},\alpha}}~\cdots~\overline{u_n\phi_{\alpha_n,\alpha}}\\
                                               & =
\overline{u_1\phi_{\alpha_1, \alpha}}~\cdots~\overline{u_{i-1}\phi_{\alpha_{i-1},\alpha}}~\overline{x\phi_{\delta, \alpha}y\phi_{\eta,\alpha}}~\overline{u_{i+1}\phi_{\alpha_{i+1},\alpha}}~\cdots~\overline{u_n\phi_{\alpha_n,\alpha}}\\
                                              & =
\overline{u_1\phi_{\alpha_1, \alpha}}~\cdots~\overline{u_{i-1}\phi_{\alpha_{i-1},\alpha}}~\overline{x\phi_{\delta, \alpha}}~\overline{y\phi_{\eta,\alpha}}~\overline{u_{i+1}\phi_{\alpha_{i+1},\alpha}}~\cdots~\overline{u_n\phi_{\alpha_n,\alpha}}
\end{aligned}
$$ as required.\end{proof}

\begin{coro}\label{equal-coro}
Let $B$ be a normal band and let $x_1, \cdots,x_n, y_1,\cdots,y_m\in B_\alpha$. Then $\overline{x_1}~\cdots~\overline{x_n}=\overline{y_1}~\cdots~\overline{y_m}$ in $\ig(B_\alpha)$ if and only if the equality holds in $\ig(B)$.
\end{coro}
\begin{proof}
The necessity is obvious, as any basic pair in $B_\alpha$ must also be basic in $B$. Suppose now that we have $$\overline{x_1}~\cdots~\overline{x_n}=\overline{y_1}~\cdots~\overline{y_m}$$ in $\ig(B).$  Then there exists a sequence 
$$\overline{x_1}~\cdots~\overline{x_n}\sim\overline{u_1}~\cdots~\overline{u_s}\sim \overline{v_1}~\cdots~\overline{v_t}\sim~\cdots~\sim \overline{w_1}~\cdots~\overline{w_l}\sim\overline{y_1}~\cdots~\overline{y_m}.$$ 
Note that all idempotents involved in the above sequence lie in  components $B_\beta$ where $\beta\geq \alpha,$ so that successive applications of Lemma \ref{equal} give $\overline{x_1}~\cdots~\overline{x_n}=\overline{y_1}~\cdots~\overline{y_m}$ in $\ig(B_\alpha).$
\end{proof}

We remark here that for an arbitrary band $B$, Corollary \ref{equal-coro} need not be true.

\begin{ex} {\rm Let $B=B_{\alpha}\cup B_{\beta}$ be a band with semilattice structure and multiplication table defined by
\[\begin{array}{cc}
\begin{array}{c|ccccc}
  & l & u & w & u' & w'\\ \hline
l & l & u' & w' & u' & w' \\
u & u & u & w & u & w\\
w & w & u & w & u & w \\
u' & u' & u' & w' & u' & w'\\
w' & w' & u' & w' & u' & w'
\end{array}&
\begin{tikzpicture}
\node (a) at (-1,2) {$B_\alpha$};
\node (y) at (0,2) {\boxed{l}};
\node (b) at (-1.5,0) {$B_\beta$};
\node (x) at (0,0) {\begin{tabular}{|l|c|}
\hline
$u'$&$w'$\\
\hline
$u$&$w$\\
\hline
\end{tabular}};
\path[-,font=\scriptsize,>=angle 60]
(y) edge node[above]{} (x);
\end{tikzpicture}
\end{array}\]

It is easy to check that $B$ forms a band. By the uniqueness of normal forms in $\ig(B_\beta),$ we have $\overline{u'}~\overline{w}\neq \overline{w'}$ in $\ig(B_\beta)$. However in $\ig(B)$ we have
$$
\begin{aligned}
\overline{u'}~\overline{w}&=\overline{u'l}~\overline{w}\\
                                              &=
\overline{u'}~\overline{l}~\overline{w} \ \ \ \ (\mbox{as~} (u',l) \mbox{~is a basic pair})\\
                                              &=
\overline{u'}~\overline{lw} \ \ \ \ (\mbox{as~} (l,w) \mbox{~is a basic pair})\\
                                              &=
\overline{u'}~\overline{w'}\\
                                              &=
\overline{w'}
\end{aligned}
$$}
\end{ex}
With the above preparations, we now construct a 10-element normal band $B$ for which $\ig(B)$ is not abundant.
\begin{ex}\label{ex:final} {\rm Let $B=\mathcal{B}(Y; B_{\alpha}, \phi_{\alpha,\beta})$ be a strong semilattice $Y=\{\alpha,\beta,\gamma,\delta\}$ of rectangular bands (see the figure below), where $\phi_{\alpha, \beta}: B_{\alpha}\longrightarrow B_\beta$ is defined by $$a\phi_{\alpha, \beta}=e, b\phi_{\alpha, \beta}=f, c\phi_{\alpha, \beta}=g, d\phi_{\alpha, \beta}=h$$ the remaining morphisms being defined in the obvious unique manner.
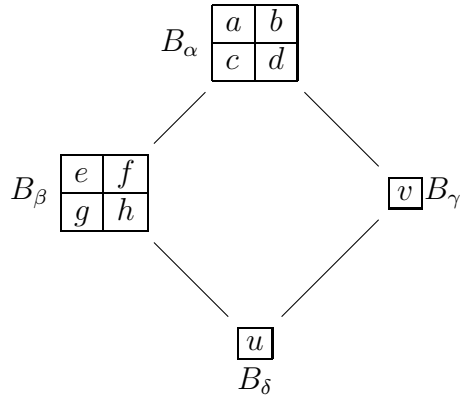
\begin{center}
\begin{tikzpicture}[scale=1.0]
\node (a) at (-1,2.0) {$B_\alpha$};
\node (w) at (0,2) {\begin{tabular}{|l|c|}
\hline
$a$&$b$\\
\hline
$c$&$d$\\
\hline
\end{tabular}};
\node (b) at (-3,0) {$B_\beta$};
\node (x) at (-2,0) {\begin{tabular}{|l|c|}
\hline
$e$&$f$\\
\hline
$g$&$h$\\
\hline
\end{tabular}};
\node (y) at (2,0) {\boxed{v}};
\node (c) at (2.5,0) {$B_\gamma$};
\node (z) at (0,-2) {\boxed{u}};
\node (a) at (0,-2.5) {$B_\delta$};
\path[-,font=\scriptsize,>=angle 60]
(w) edge node[above]{} (x)
(w) edge node[right]{} (y)
(x) edge node[right]{} (z)
(y) edge node[above]{} (z);
\end{tikzpicture}
\end{center}
Considering the element $\overline{e}~\overline{v}\in \ig(B)$,  we have
$$
\begin{aligned}
\overline{e}~\overline{v}&=\overline{e}~\overline{dv}\\
                                              &=
\overline{e}~\overline{d}~\overline{v} \ \ \ \ (\mbox{as~} (d,v) \mbox{~is a basic pair})\\
                                              &=
\overline{e}~\overline{h}~\overline{v} \ \ \ \ (\mbox{as~} \overline{e}~\overline{d}=\overline{e}~\overline{d\phi_{\alpha,\beta}}=\overline{e}~\overline{h} \mbox{~by~Corollary~\ref{almost normal form-simple})}\\
                                              &=
\overline{e}~\overline{h}~\overline{av}\\
                                              &=
\overline{e}~\overline{h}~\overline{a}~\overline{v} \ \ \ \ (\mbox{as~} (a,v) \mbox{~is a basic pair})\\
                                              &=
\overline{e}~\overline{h}~\overline{e}~\overline{v} \ \ \ \ (\mbox{as~} \overline{h}~\overline{a}=\overline{h}~\overline{a\phi_{\alpha,\beta}}=\overline{h}~\overline{e} \mbox{~by~Corollary~\ref{almost normal form-simple})}
\end{aligned}
$$
However, $\overline{e}~\overline{h}~\overline{e}\neq \overline{e}$ in $\ig(B_\beta)$ by the uniqueness of normal forms, so by Corollary \ref{equal-coro}, we have $\overline{e}~\overline{h}~\overline{e}\neq e$ in $\ig(B)$, which implies $\overline{e}~\overline{v}$ is not $\mathcal{R}^*$-related to $\overline{e}$. On the other hand, we have known from Theorem \ref{IG(band)} that $\overline{e}~\overline{v}~\mathcal{\widetilde{R}}~\overline{e},$ so that by Lemma \ref{observation1} that $\overline{e}~\overline{v}$ is not $\mathcal{R}^*$-related any idempotent of $B$, so that $\ig(B)$ is not an abundant semigroup.}
\end{ex}

\end{document}